\documentclass{amsart}
\usepackage{hyperref}
\usepackage{fullpage}
\usepackage{amsrefs}
\usepackage{amssymb}
\usepackage{amsbsy}

\input xypic
\newdir{ >}{{}*!/-9pt/\dir{>}}

\newcommand{\Ch}        {\operatorname{Ch}}
\newcommand{\Hom}       {\operatorname{Hom}}
\newcommand{\Map}       {\operatorname{Map}}

\newcommand{\res}	{\operatorname{res}}
\newcommand{\spec}      {\operatorname{spec}}
\newcommand{\sSet}	{\operatorname{sSet}}

\newcommand{\img}       {\operatorname{image}}
\newcommand{\sgn}	{\operatorname{sgn}}
 
\newcommand{\N}         {{\mathbb{N}}}
\newcommand{\Z}         {{\mathbb{Z}}}
\newcommand{\Q}         {{\mathbb{Q}}}
\newcommand{\R}         {{\mathbb{R}}}
\newcommand{\K}	        {{\mathbb{K}}}

\newcommand{\al}        {\alpha}
\newcommand{\bt}        {\beta} 
\newcommand{\gm}        {\gamma}
\newcommand{\dl}        {\delta}
\newcommand{\ep}        {\epsilon}
\newcommand{\zt}        {\zeta}
\newcommand{\tht}       {\theta}
\newcommand{\lm}        {\lambda}
\newcommand{\sg}        {\sigma}
\newcommand{\up}        {\upsilon}
\newcommand{\om}        {\omega}

\newcommand{\Dl}        {\Delta}
\newcommand{\Sg}        {\Sigma}
\newcommand{\Lm}        {\Lambda}
\newcommand{\Om}        {\Omega}
\newcommand{\Tht}       {\Theta}

\newcommand{\Sgi}       {\Sigma^\infty}
\newcommand{\Smash}     {\wedge}
\newcommand{\bigSmash}  {\bigwedge}
\newcommand{\colim}  {\operatornamewithlimits{\underset{\longrightarrow}{lim}}}
\newcommand{\ip}[1]     {\langle #1\rangle}
\newcommand{\iip}[1]    {\boldsymbol{(}#1\boldsymbol{)}}
\newcommand{\ot}        {\otimes}
\newcommand{\op}        {\oplus}
\newcommand{\sm}        {\setminus}
\newcommand{\sse}       {\subseteq}
\newcommand{\st}        {\;|\;}
\newcommand{\tm}        {\times}
\newcommand{\DDl}	{\mathbf{\Delta}}

\newcommand{\tN}	{\widetilde{N}}

\newcommand{\tC}	{\widetilde{C}}

\newcommand{\tH}        {\widetilde{H}}
\newcommand{\tP}	{\widetilde{P}}
\newcommand{\tW}	{\widetilde{W}}

\newcommand{\hPhi}	{\widehat{\Phi}}

\newcommand{\ttht}      {\widetilde{\theta}}
\newcommand{\un}[1]     {\underline{#1}}
\newcommand{\ov}[1]     {\overline{#1}}
\newcommand{\CV}        {{\mathcal{V}}}
\newcommand{\xla}       {\xleftarrow}
\newcommand{\xra}       {\xrightarrow}
\newcommand{\hx}	{\widehat{x}}
\newcommand{\hy}	{\widehat{y}}

\newcommand{\EE}        {{\mathbb{E}}}
\newcommand{\half}      {{\textstyle\frac{1}{2}}}
\newcommand{\opp}	{\text{op}}
\newcommand{\im}        {\vdash}

\newcommand{\CA}        {{\mathcal{A}}}
\newcommand{\CS}        {{\mathcal{S}}}
\newcommand{\CX}        {{\mathcal{X}}}
\newcommand{\ND}	{\operatorname{ND}}

\renewcommand{\:}{\colon}

\newtheorem{theorem}{Theorem}[section]

\newtheorem{lemma}[theorem]{Lemma}
\newtheorem{proposition}[theorem]{Proposition}
\newtheorem{corollary}[theorem]{Corollary}
\theoremstyle{definition}
\newtheorem{remark}[theorem]{Remark}
\newtheorem{definition}[theorem]{Definition}

\newtheorem{construction}[theorem]{Construction}

\begin{document}
\title{Chains on suspension spectra}
\author{N.~P.~Strickland}

\maketitle 

\begin{abstract}
 We define and study a homological version of Sullivan's rational de
 Rham complex for simplicial sets.  This new functor can be
 generalised to simplicial symmetric spectra and in that context it
 has excellent categorical properties which promise to make a number
 of interesting applications much more straightforward.
\end{abstract}

\section{Introduction}\label{sec-intro}

In this paper we will define and study a functor $\Phi$ from
simplicial sets to rational chain complexes, with the property that
$H_*(\Phi_*(X))$ is just the ordinary rational homology of $X$.

Some background is needed to understand why this functor deserves
attention.  There is a much simpler functor called $N_*$ (normalised
simplicial chains) from simplicial sets to integral chain complexes
that computes integral homology, and one can just tensor with $\Q$ to
compute rational homology.  There is a dual complex $N^*$ that
calculates integral cohomology.  This is equipped with a natural
product $N^*(X)\ot N^*(X)\to N^*(X)$ which is commutative up to
homotopy but not on the nose.  The theory of Steenrod operations shows
that if we work integrally then neither $N^*(X)$ nor any reasonable
replacement can be given a strictly commutative product (even with the
usual signs).  Rationally, however, the situation is better:
in~\cite{su:ict} Sullivan developed a rational and simplicial version
of de Rham theory giving a cochain complex $\Om^*(X)$ with a strictly
commutative product that computes the ordinary rational cohomology of
$X$.  This can be used as a starting point for the rich and powerful
theory of rational homotopy (originally introduced by
Quillen~\cite{qu:rht} using slightly different machinery).  One can
then stabilise and consider the category $\CS_\Q$ of rational spectra,
which makes things considerably simpler: it is well-known that the
homotopy category of $\CS_\Q$ is equivalent to the category of graded
rational vector spaces.  However, we can make things harder again by
considering rational spectra with a ring structure or a group action.
To handle these, we need to improve the homotopy classification of
rational spectra to some kind of monoidal Quillen equivalence of
$\CS_\Q$ with a suitable model category $\Ch_\Q$ of rational chain
complexes.

Work of this type has been done especially by Greenlees, Shipley and
Barnes, leading to very concrete and interesting descriptions of the
homotopy theory of $G$-spectra for various compact Lie groups $G$,
among other things.  However, some of the arguments involved are more
awkward than one might like, because they do not have a single
symmetric monoidal Quillen functor $\Psi_*\:\CS_\Q\to\Ch_\Q$, but a
zig-zag of Quillen functors whose monoidal properties fit together in
an inconvenient way.

Recently, the author and Stefan Schwede independently discovered a
functor $\Psi_*$ as above, which promises to simplify many
applications such as those of Greenlees \emph{et al}.  This will be
explained in a separate paper by Schwede and the present author.  It
is then natural to ask for a calculation of $\Psi_*(T)$ for various
popular spectra $T$, including suspension spectra.  One of the most
intriguing aspects of the story is that the complex
$\Phi_*(X)=\Psi_*(\Sgi X_+)$ has a very natural description in terms
of simplicial de Rham theory, although nothing of that kind is visible
in the definition.  In particular, we obtain a chain complex similar
in spirit to $\Om^*(X)$ that computes $H_*(X;\Q)$ rather than
$H^*(X;\Q)$; this cannot reasonably be done by naive dualisation, as
$\Om^*(X)$ is infinite-dimensional (even when $X$ is finite) and has
no natural topology.  This forms the main subject of the present
paper.

It will be convenient for us to work in a slightly different order
from that suggested by the above discussion.  We will give a
definition of $\Phi_*(X)$ that does use de Rham theory, and
investigate the properties of $\Phi$ using that definition.
Eventually, in Theorem~\ref{thm-Phi-colim} we will obtain a
description of $\Phi_*(X)$ as a colimit of groups that do not involve
differential forms.  When we have defined $\Psi$ (in a separate paper)
it will be clear from that description that
$\Psi_*(\Sgi X_+)=\Phi_*(X)$.

Appendix~\ref{apx-simplicial} contains some recollections and
notational conventions about the simplicial category (especially the
theory of shuffles) which will be in place throughout the paper.
Appendix~\ref{apx-int} contains formulae for integrals of polynomials
over simplices.  These are surely standard, but we do not know a
convenient source.

\section{de Rham chains}\label{sec-de-rham}

Let $\K$ be a field of characteristic zero.  Some of our constructions
will seem most natural for $\K=\Q$ and others for $\K=\R$, but in fact
everything works for any $\K$.

Given a finite set $I$, we put 
\begin{align*}
 \tP_I &= \K[t_i\st i\in I] \\
 P_I &= \tP_I/(1-\sum_it_i),
\end{align*}
so $P_I$ is the ring of polynomial functions on an algebraic simplex
$\Dl^{\text{alg}}_I=\spec(P_I)$ of dimension $|I|-1$.  We also put
\begin{align*}
 W_I &= \K\{dt_i\st i\in I\}/(\sum_i dt_i) \\
 \Om^1_I &= P_I\ot_\K W_I = P_I\{dt_i\st i\in I\}/(\sum_i dt_i) \\
 \Om^*_I &= P_I\ot_\K\Lm^*(W_I) = \Lm^*_{P_I}(\Om^1_I).
\end{align*}
Here $\Om^*_I$ is graded with $|t_i|=0$ and $|dt_i|=1$, and we give
$\Om^*_I$ the standard de Rham differential, making it a differential
graded algebra.  All of these constructions are contravariantly
functorial in $I$: a map $\al\:I\to J$ of finite sets gives a ring map
$\al^*\:P_J\to P_I$ with $\al^*(t_j)=\sum_{\al(i)=j}t_i$, and this
extends naturally to a map $\al^*\:\Om^*_J\to\Om^*_I$.  If $\al$ is
just the inclusion of a subset, we write $\res^J_I$ for $\al^*$.

In particular, the assignment $n\mapsto\Om^*_{[n]}$ is a simplicial
object in the category of DGA's, so for any simplicial set $X$ we can
define 
\[ \Om^k(X) = \sSet(X,\Om^k_\bullet) \]
and this gives us a differential graded algebra $\Om^*(X)$.  It is
well-known that $H^*\Om^*(X)$ is the usual cohomology $H^*(X;\K)$.

We would like a version of this construction that is well-related to
homology rather than cohomology.  The most obvious approach is to
dualise and put 
\[ \hPhi_{I,k} = \Hom_\K(\Om^k_I,\K), \] 
giving a chain complex that is covariantly functorial
in $I$.  However, this is inconvenient because $\hPhi^k_I$ is most
naturally a product (rather than direct sum) of countably many copies
of $\Q$, which introduces numerous technical complications.  We will
therefore use a smaller subcomplex $\Phi_{I,*}\leq\hPhi_{I,*}$.

\begin{definition}\label{defn-Phi}
 We define
 \begin{align*}
  W_I^\vee &= \Hom_\K(W_I,\K) \\
  \Tht_{I,m} &= P_I\ot\Lm^m(W_I^\vee) = \Lm^m_{P_I}(P_I\ot W_I^\vee) \\
  \Phi_{I,m} &= \bigoplus_{\emptyset\neq J\sse I} \Tht_{J,m}.
 \end{align*}
 We write $i_J$ for the inclusion $\Tht_{J,m}\to\Phi_{I,m}$.  We will
 occasionally use a bigrading on $\Phi_{I,*}$: we put
 \[ \Phi_{I,(p,q)} = \bigoplus_{|J|=p} \Tht_{J,p+q} \]
 so that $\Phi_{I,m}=\bigoplus_{p+q=m}\Phi_{I,(p,q)}$.
\end{definition}

We want to interpret $\Phi_{I,*}$ as a subcomplex of $\hPhi_{I,*}$,
and for this we need to define various bilinear pairings.
First, we define a pairing of $\Lm^m(W^\vee_I)$ with $\Lm^m(W_I)$ by 
the formula
\[ \ip{\al_1\wedge\dotsb\wedge \al_m,\om_1\wedge\dotsb\wedge \om_m}_I
     = (-1)^{m(m-1)/2}\det(\ip{\al_i,\om_j})_{i,j=1}^m.
\]
This is a perfect pairing, and we will silently use it to identify
$\Lm^m(W_I^\vee)$ with $\Lm^m(W_I)^\vee$.  Next, we can extend this
linearly over $P_I$ to get a pairing 
\[ \ip{\cdot,\cdot}_I \: \Tht_{I,m}\ot \Om_I^m\to P_I \]
given by essentially the same formula.  Occasionally we will use the
convention $\ip{\al,\om}=0$ if $\al\in\Tht_{I,m}$ and $\om\in\Om_I^p$
with $p\neq m$.
\begin{remark}\label{rem-sign}
 The factor $(-1)^{m(m-1)/2}$ is inserted to ensure that the term
 $\prod_i\ip{\al_i,\om_i}$ in the determinant comes with the standard
 sign for converting the term
 \[ \al_1\ot\dotsb\ot\al_m\ot\om_1\ot\dotsb\ot\om_m \]
 to the term 
 \[ \al_1\ot\om_1\ot\al_2\ot\om_2\ot\dotsb\ot\al_m\ot\om_m. \]
 In other words, if we defined the pairing by a diagram in the usual
 notation of symmetric monoidal categories, then the sign would come
 from the twist maps and so would not need to be inserted explicitly.
\end{remark}

We really want a pairing with values in $\K$ rather than $P_I$, and
for this we need to integrate.  
\begin{definition}\label{defn-int}
 Given a monomial $t^\nu=\prod_{i\in I}t_i^{\nu_i}$, we put $n=|I|-1$
 and define
 \[ \int_I t^\nu = 
     \left(\prod_i \nu_i!\right)/(n+\sum_i\nu_i)! \in\K.
 \]
 This extends to a linear map $\int_I\:\tP_I\to\K$, and one can check
 (see Lemma~\ref{lem-int-well-defined}) that it factors through the
 quotient $P_I=\tP_I/(1-\sum_it_i)$.  It is often convenient to use
 the notation $\nu!=\prod_k(\nu_k!)$ and $t^{[\nu]}=t^\nu/\nu!$ and
 $|\nu|=\sum_i\nu_i$, so that $\int_It^{[\nu]}=1/(n+|\nu|)!$.
\end{definition}
\begin{remark}\label{rem-int}
 One can also check (see Lemma~\ref{lem-int-real}) that in the case
 $\K=\R$, the map $\int_I\:P_I\to\R$ is just integration over the
 simplex $\Dl_I$ with respect to a natural measure.
\end{remark}
\begin{remark}\label{rem-int-forms}
 There is a theory of integration for functions on a space with a
 measure, and also a theory of integration for differential forms on a
 manifold with orientation.  In discussing de Rham cohomology it is
 more usual to use integration of forms, but in our application it is
 painful to keep track of the orientations, so we have chosen to
 reformulate everything in terms of integration of functions.
\end{remark}

\begin{definition}\label{defn-xi}
 We define a pairing $\iip{\cdot,\cdot}\:\Phi_{I,m}\ot\Om_I^m\to\K$
 by $\iip{i_J(\al),\om}=\int_J\ip{\al,\res^I_J(\om)}$.  In particular,
 for $\al\in\Tht_{I,m}\leq\Phi_{I,m}$ we just have
 $\iip{\al,\om}=\int_I\ip{\al,\om}$.  We let
 $\xi\:\Phi_{I,m}\to\hPhi_{I,m}$ be adjoint to $\iip{\cdot,\cdot}$.
\end{definition}

Our main results about $\Phi$ are summarised below; proofs will be
given in the subsequent sections of the paper.
\begin{theorem}\label{thm-Phi}
 \begin{itemize}
  \item[(a)] The map $\xi_I$ is injective, and the image (which we
   will identify with $\Phi_{I,*}$) is a subcomplex of $\hPhi_{I,*}$.
  \item[(b)] $\Phi_{I,*}$ is a covariant functor of $I$, and the maps
   $\al_*\:\Phi_{I,*}\to\Phi_{J,*}$ are quasiisomorphisms.
  \item[(c)] For the singleton $1=\{0\}$ we have $\Phi_{1,*}=\Q$
   (concentrated in degree zero). \qed
 \end{itemize}
\end{theorem}

\begin{definition}\label{defn-PhiX}
 If $X$ is a simplicial set, we let $\Phi_*(X)$ be the coend of the
 functor $\DDl^{\opp}\tm\DDl\to\Ch_\K$ given by
 $(n,m)\mapsto\Z[X_n]\ot\Phi_{[m],*}$.
\end{definition}

\begin{theorem}\label{thm-PhiX}
 $\Phi$ is a lax symmetric monoidal functor from spaces to chain
 complexes, with a natural isomorphism $H_*\Phi_*(X)=H_*(X;\K)$.
 There is a natural $\K$-linear isomorphism
 \[ \Phi_d(X) = \bigoplus_k N_k(X) \ot \Tht_{[k],d}, \]
 where $N_*(X)$ is the group of normalised chains on $X$. \qed
\end{theorem}

\begin{theorem}\label{thm-Phi-colim}
 There is a natural isomorphism 
 \[ \Phi_*(X) = \colim_A\Hom(\tH_*(S^A),\tN_*(S^A\Smash X_+)), \]
 where $A$ runs over the category of finite sets and injective
 maps. \qed 
\end{theorem}

\section{The differential}

We next introduce a differential $\dl\:\Phi_{I,m+1}\to\Phi_{I,m}$.
This involves interior multiplication, which we now recall.
\begin{definition}
 Let $U$ be a finitely generated free module over a ring $R$, with
 dual $U^\vee=\Hom_R(U,R)$.  Given $u\in U$ and
 $a\in\Lm^{k+1}(U^\vee)$, we let $u\im a\in\Lm^k(U^\vee)$ denote the
 unique element such that
 \[ \ip{u\im a,v} = (-1)^{k+1}\ip{a,u\wedge v}
     \hspace{4em} \text{ for all } v\in\Lm^k(U)
 \]
 (using the standard pairings described in
 Section~\ref{sec-de-rham}). 
\end{definition}

\begin{lemma}\label{lem-int-mult}
 \begin{itemize}
  \item[(a)] If $a\in U^\vee=\Lm^1(U^\vee)$ we have
   $u\im a=-\ip{u,a}$.
  \item[(b)] If $a\in\Lm^p(U^\vee)$ and $b\in\Lm^q(U^\vee)$ then
   $u\im(a\wedge b)=(u\im a)\wedge b+(-1)^pa\wedge(u\im b)$.
  \item[(c)] If $u,v\in U$ and $a\in\Lm^k(U^\vee)$ then
   $u\im(v\im a)+v\im(u\im a)=0$.
  \item[(d)] If $a\in\Lm^{k+1}(U)$ then
   $u\im a\in\Lm^k((U/u)^\vee)\leq\Lm^k(U^\vee)$.  Moreover, there is
   a well-defined multiplication
   $u\wedge(\cdot)\:\Lm^k(U/u)\to\Lm^{k+1}(U)$ and in this context we
   again have $\ip{u\im a,v}=(-1)^{k+1}\ip{a,u\wedge v}$.
 \end{itemize}
\end{lemma}
\begin{proof}
 This is fairly standard multilinear algebra and is left to the
 reader. 
\end{proof}

\begin{definition}\label{defn-delta}
 Suppose we have $\emptyset\neq J\sse I$ and $f\in P_J$ and
 $\al_0\in\Lm^d(W_J^\vee)$, so $i_J(f\,\al_0)\in\Phi_{I,d}$.  Note
 that we have an interior product
 $\Om^1_J\ot_{P_J}\Tht_{J,d}\to\Tht_{J,d-1}$, so we can interpret
 $df\im\al_0$ as an element of $\Tht_{J,d-1}$.  Also, if $j\in J$ we can
 interpret $dt_j\im\al_0$ as an element of
 $\Lm^{d-1}((W_J/dt_j)^\vee)=\Lm^{d-1}(W^\vee_{J\sm\{j\}})$.  We can
 thus put 
 \begin{align*}
  \dl'(i_J(f\,\al_0)) &= -i_J(df\im\al_0) \\
   &= -\sum_{j\in J} i_J((\partial f/\partial t_j)\, dt_j\im\al_0) \\ 
  \dl''(i_J(f\,\al_0)) &= 
     - \sum_{j\in J}
        i_{J\sm\{j\}}(\res^J_{J\sm\{j\}}(f)\,dt_j\im\al_0) \\
  \dl(\al) &= \dl'(\al) + \dl''(\al).
 \end{align*}
 (Here the second description of $\dl'(i_J(f\al_0))$ relies on the
 choice of a lift of $f\in P_J$ to $\tP_J$, but the first description
 shows that the result is independent of the lift.)  This gives maps
 \[ \xymatrix{
     \Phi_{I,(p,q)} \rto^{\dl'} \dto_{\dl''} &
     \Phi_{I,(p,q-1)} \dto^{\dl''} \\
     \Phi_{I,(p-1,q)} \rto_{\dl'} &
     \Phi_{I,(p-1,q-1)}
    }
 \]
 and thus $\dl\:\Phi_{I,m}\to\Phi_{I,m-1}$.  We will show that the
 square above anticommutes.
\end{definition}

\begin{proposition}
 We have $\dl'\dl''+\dl''\dl'=0$ and $(\dl')^2=0$ and $(\dl'')^2=0$
 and $\dl^2=0$, so that $\Phi_{I,(*,*)}$ is a double complex.
\end{proposition}
\begin{proof}
 The first three equations follow directly from the definitions, using
 the second description of $\dl'$, the commutation of partial
 derivatives and the rule $u\im(v\im a)+v\im(u\im a)=0$.  We can then
 expand out $(\dl'+\dl'')^2$ to see that $\dl^2=0$.
\end{proof}

\begin{proposition}\label{prop-adjoint}
 The map $\xi\:\Phi_{I,*}\to\hPhi_{I,*}$ is a chain map.
 Equivalently, for $\al\in\Phi_{I,d+1}$ and $\om\in\Om^d_I$ we have
 \[ \iip{\dl(\al),\om} = (-1)^{d+1}\iip{\al,d\om}. \]
\end{proposition}

In order to prove this, we need a definition and a lemma.

\begin{definition}\label{defn-grad}
 For any vector $x\in \K^I$ we write $\nabla_x$ for the operator
 $\sum_ix_i\frac{\partial}{\partial t_i}$ on $\tP_I$.  We note that
 this induces an operation on $P_I=\tP_I/(1-\sum_it_i)$ iff
 $\sum_ix_i=0$.
\end{definition}

\begin{lemma}\label{lem-grad}
 For $f\in P_I$ and $\sum_ix_i=0$ we have 
 \[ \int_I \sum_i \nabla_x f +
     \sum_i x_i\int_{I\sm\{i\}} \res^I_{I\sm\{i\}} f = 0.
 \]
\end{lemma}
(This is a version of Stokes's Theorem, but it is easier to prove it
directly than to do the translation necessary to quote it from
elsewhere.) 
\begin{proof}
 It will suffice to prove this for a monomial $f=t^{[\nu]}$.  Put
 $\ep=1/(|\nu|+n-1)$ and $J=\{i\in I\st\nu_i>0\}$, and suppose that
 $i\in J$.  Let $\dl_i\:I\to\{0,1\}$ be the Kronecker delta, so
 $\partial f/\partial t_i=t^{[\nu-\dl_i]}$ and $|\nu-\dl_i|=|\nu|-1$.
 We then have $\int_Ix_i\partial f/\partial t_i=x_i\ep$, but
 $\res^I_{I\sm\{i\}}f=0$. 
 Suppose instead that $i\not\in J$.  Then $\partial f/\partial t_i=0$ but
 $\int_{I\sm\{i\}} \res^I_{I\sm\{i\}} f=
 \int_{I\sm\{i\}}t^{[\nu]}=\ep$.  Thus the first term in the claimed
 equation is $\sum_{i\in J}x_i\ep$, and the second term is
 $\sum_{i\not\in J}x_i\ep$, so altogether we have $\ep.\sum_Ix_i=0$.
\end{proof}

\begin{lemma}\label{lem-adjoint-special}
 Proposition~\ref{prop-adjoint} holds when
 $\al\in\Tht_{I,d+1}\leq\Phi_{I,d+1}$. 
\end{lemma}
\begin{proof}
 We reduce by linearity to the case where $\al=f\,\al_0$ and
 $\om=g\,\om_0$ for some $f,g\in\tP_I$ and
 $\al_0\in\Lm^{d+1}(W_I^\vee)$ and $\om_0\in\Lm^d(W_I)$.  Put
 \[ x_i=\ip{dt_i\im\al_0,\om_0}=
     (-1)^{d+1}\ip{\al_0, dt_i\wedge\om_0} \in\K,
 \]
 and observe that $\sum_ix_i=0$ (because $\sum_idt_i=0$).  We can thus
 apply Lemma~\ref{lem-grad} to the function $fg$ giving
 \[ \int_I f.\nabla_x(g) + \int_I \nabla_x(f).g + 
     \sum_i x_i\int_{I\sm\{i\}}\res^I_{I\sm\{i\}}(fg) = 0.
 \]
 From the definitions we find that
 \begin{align*}
  f.\nabla_x(g) 
   &= (-1)^{d+1}\sum_i
       f\tfrac{\partial g}{\partial t_i}\ip{\al_0,dt_i\wedge\om_0} \\
   &= (-1)^{d+1}\ip{f\al_0,dg\wedge\om_0} 
    = (-1)^{d+1}\ip{\al,d\om}.
 \end{align*}
 By a similar argument, we have $\nabla_x(f)g=\ip{df\im\al_0,\om}$.
 Next, recall that we can interpret $dt_i\im\al_0$ as an element of
 $\Lm^d(W_{I\sm\{i\}}^\vee)$, and then we have
 \[ x_i = \ip{dt_i\im\al_0,\res^I_{I\sm\{i\}}(\om_0)}. \]
 It follows that 
 \[ x_i\res^I_{I\sm\{i\}}(fg) = 
     \ip{\res^I_{I\sm\{i\}}(f)dt_i\im\al,\res^I_{I\sm\{i\}}(\om)},
 \]
 and thus that 
 \[ \int_{I\sm\{i\}} x_i\res^I_{I\sm\{i\}}(fg) = 
     \iip{i_{I\sm\{i\}}(\res^I_{I\sm\{i\}}(f)dt_i\im\al_0),\om}. 
 \]
 The lemma now follows by combining these facts with the definition of
 $\dl(\al)$.  
\end{proof}

\begin{proof}[Proof of Proposition~\ref{prop-adjoint}]
 The element $\al\in\Phi_{I,m+1}$ can be written as
 $\sum_{\emptyset\neq J\sse I}i_J(\al_J)$, with $\al_J\in\Tht_J$.  By
 applying Lemma~\ref{lem-adjoint-special} to the pairs
 $(\al_J,\res^I_J(\om))$ we recover the statement of
 Proposition~\ref{prop-adjoint}.
\end{proof}

\begin{lemma}\label{lem-xi-inj}
 The map 
 \[ \xi_I \: \Phi_{I,k} \to \hPhi_{I,k}. \]
 is injective.
\end{lemma}
\begin{proof}
 If we can prove this for $\K=\R$ then it will follow for $\K=\Q$ by
 restriction, and then for arbitrary $\K$ by tensoring up again.  We
 therefore take $\K=\R$ for the rest of the proof.

 Consider a nonzero element $\al=\sum_Ji_J(\al_J)$ of the
 domain.  Choose a set $J$ of largest possible size with
 $\al_J\neq 0$ in $\Tht_{J,k}$.  As $\al_J$ is nonzero, and
 $\Tht_{J,k}$ is dual over $P_J$ to  $\Om_J^k$, and the
 restriction map $\Om_I^k\to\Om_J^k$ is surjective, we can choose
 $\om\in\Om^k_I$ such that the element
 $f_0=\ip{\al_J,\res^I_J(\om)}\in P_J$ is nonzero.  We can then choose
 $f\in P_I$ with $\res^I_J(f)=f_0$.    
 We also put $g=\prod_{j\in J}t_j\in P_I$ and $\tht=fg\om\in\Om_I^k$.
 We claim that $\xi_I(\al)(\tht)=\iip{\al,\tht}\neq 0$.  Indeed,
 we have 
 \[ \iip{i_J(\al_J),\tht} =
     \int_J \ip{\al_J,\res^I_J(fg\om)}
     = \int_J f_0^2\res^I_J(g).
 \]
 Now $g>0$ on the interior of the simplex $\Dl_J$, and $f_0^2$ is
 nonnegative everywhere and strictly positive on a nonempty open set,
 so the integral is strictly positive.  However, we also need to
 consider the other terms $\iip{i_K(\al_K),\tht}$ for $K\neq J$.   
 If $K$ is a strict superset of $J$ then $\al_K=0$ by our choice of
 $J$.  If $K\not\supseteq J$  then we can
 choose $j\in J\sm K$ and then $\res^I_K(t_j)=0$ so $\res^I_K(g)=0$.
 Either way we find that $\iip{i_K(\al_K),\om}=0$.  It follows that
 $\iip{\al,\om}=\iip{i_J(\al_J),\om}>0$, as required.
\end{proof}

\begin{definition}\label{defn-theta}
 Let $\tW_I$ be the vector space freely generated by
 $\{dt_i\st i\in I\}$, so $W_I=\tW_I/\sum_idt_i$.  Let
 $\{e_i\st i\in I\}$ be the obvious basis for $\tW_I^\vee$, so that
 $W_I^\vee$ is spanned by the elements $e_i-e_j$.  Next, in the case
 $I=[n]=\{0,1,\dotsc,n\}$ put 
 \begin{align*}
  \ttht_{[n]} &=
    e_0\wedge e_1\wedge\dotsb\wedge e_n
     \in\Lm^{n+1}(\tW_{[n]}^\vee) \\
  \tht_{[n]}  &= 
    (e_1-e_0)\wedge(e_2-e_0)\wedge\dotsb\wedge(e_n-e_0) \\
  &= (e_1-e_0)\wedge(e_2-e_1)\wedge\dotsb\wedge(e_n-e_{n-1})
      \in\Lm^n(W_{[n]}^\vee)\leq\Tht_{[n],n}.
 \end{align*}
 It is an exercise to check that the two expressions for $\tht_{[n]}$
 are the same, and that $e_i\wedge\tht_{[n]}=\ttht_{[n]}$ for all $i$,
 and that $\tht_{[n]}$ is the unique element of $\Lm^n(W_{[n]}^\vee)$
 with this property.

 If $I$ is any finite ordered set with $|I|=n+1$ then there is a
 unique ordered bijection $[n]\to I$, and we use this to define
 $\ttht_I\in\Lm^{n+1}(\tW_I^\vee)$ and $\tht_I\in\Lm^n(W_I^\vee)$.  It
 is easy to see that $\Lm^{n+1}(\tW_I)=\K.\ttht_I$ and
 $\Lm^n(W_I)=\K.\tht_I$.  
\end{definition}

\begin{lemma}\label{lem-dl-tht}
 We have $\dl'(\tht_{[n]})=0$ and 
 \[ \dl''(\tht_{[n]}) = \dl(\tht_{[n]}) =
     -\sum_{j\in [n]} (-1)^ji_{[n]\sm\{j\}}(\tht_{[n]\sm\{j\}}).
 \]
\end{lemma}
\begin{proof}
 By inspection of the definitions, this reduces to the claim that 
 \[ dt_j\im\tht_{[n]} = (-1)^j\tht_{[n]\sm\{j\}}. \]
 For $j=0$ it is most convenient to use the expression
 \[ \tht_{[n]} =
     (e_1-e_0)\wedge(e_2-e_1)\wedge\dotsb\wedge(e_n-e_{n-1})
 \]
 and the derivation property 
 \[ dt_0\im(a\wedge b) =
     (dt_0\im a)\wedge b + (-1)^{|a|}a\wedge(dt_0\im b).
 \]
 We have $dt_0\im(e_1-e_0)=-\ip{dt_0,e_1-e_0}=1$ and
 $dt_0\im(e_{k+1}-e_k)=0$ for $k>0$.  It follows that 
 \[ dt_0\im\tht_{[n]} = 
     (e_2-e_1)\wedge(e_3-e_2)\wedge\dotsb\wedge(e_n-e_{n-1})
      = \tht_{[n]\sm\{0\}}
 \]
 as claimed.

 For $j>0$ we instead use the expression 
 \[ \tht_{[n]} = 
     (e_1-e_0)\wedge(e_2-e_0)\wedge\dotsb\wedge(e_n-e_0).
 \]
 We have $dt_j\im(e_k-e_0)=0$ for $k\neq j$, so only the term
 $dt_j\im(e_j-e_0)$ contributes, and this has a factor $(-1)^{j-1}$
 because of its position in the list.  We also have
 $dt_j\im(e_j-e_0)=-\ip{dt_j,e_j-e_0}=-1$ which gives one more sign
 change, so $dt_j\im\tht_{[n]}=(-1)^j\tht_{[n]\sm\{j\}}$ as claimed.
\end{proof}

\begin{lemma}
 For any totally ordered set $J$ we have
 $H_*(\Tht_{J,*};\dl')=\K.\tht_J$.
\end{lemma}
(The ordering is only used here to fix the sign of the generator.)
\begin{proof}
 We may assume that $J=[m]$ for some $m$, so $P_J=\K[t_1,\dotsc,t_m]$
 and $W_J=\K\{dt_1,\dotsc,dt_m\}$.  Let $\{w_1,\dotsc,w_m\}$ be the
 dual basis for $W^\vee_J$ and put $C(i)_*=\K[t_i]\{1,w_i\}$, so that
 $\Tht_{J,*}=\bigotimes_iC(i)_*$.  It is not hard to see that this
 decomposition is compatible with the differentials, and that in
 $C(i)_*$ we have $\dl'(f(t_i)w_i)=f'(t_i)$ and $\dl'(g(t_i))=0$.  It
 follows that $H_*(C(i)_*;\dl')=\K.w_i$, and thus, by the K\"unneth
 theorem, that $H_*(\Tht_{J,*};\dl')=K.\bigwedge_iw_i=\K.\tht_J$.
\end{proof}

We can now calculate the homology of $\Phi_{I,*}$.  Note that for
$j\in I$ we have $\Tht_{\{j\},*}=\K$ (concentrated in degree zero), so
we have an element $i_{\{j\}}(1)\in\Phi_{I,0}$, which is a cycle for
degree reasons.
\begin{proposition}\label{prop-H-Phi}
 The elements $i_{\{j\}}(1)$ are all homologous to each other, and the
 corresponding homology class generates $H_0(\Phi_{I,*};\dl)$ freely
 over $\K$.  Moreover, we have $H_d(\Phi_{I,*};\dl)=0$ for all
 $d\neq 0$.
\end{proposition}
\begin{proof}
 We may assume that $I$ is totally ordered, which gives an ordering on
 each subset $J\sse I$ and thus defines elements $\tht_J$ as before.  

 We now regard $\Phi_I$ as a double complex under $\dl'$ and $\dl''$,
 and use the resulting spectral sequence.  We write $C_*$ for the
 $E_1$ page, which is just
 \[ C_* = H_*(\Phi_{I,*};\dl') =
     \K\{\tht_J\st\emptyset\neq J\sse I\}.
 \]
 The differential is given by Lemma~\ref{lem-dl-tht}.  Note also that 
 \[ \Lm^*(\tW_I^\vee)=\Lm^*(e_i\st i\in I)=\K\{\ttht_J\st J\sse I\} \]
 (and here we do have a term for $J=\emptyset$).  We can make this a
 differential graded ring with $d(e_i)=1$ for all $i$, and the
 resulting homology is zero.  We can then define
 $\phi\:\Lm^*(\tW_I^\vee)\to\Sg C_*$ by 
 $\phi(\ttht_J)=\Sg\tht_J$ when $J\neq\emptyset$, and
 $\phi(1)=\phi(\ttht_\emptyset)=0$.  It follows from
 Lemma~\ref{lem-dl-tht} that $\phi$ is a chain map.  The short exact
 sequence $\K\to\Lm^*(\tW^\vee_I)\xra{\phi}\Sg C_*$ gives a long exact
 sequence in homology.  This in turn shows that $H_i(C_*)=0$ for
 $i\neq 0$, and gives an isomorphism $H_0C_*=H_1(\Sg C_*)=\K$.  Our
 spectral sequence must therefore collapse at the $E_2$ page, so
 $H_i(\Phi_{I,*})=0$ for all $i\neq 0$, and the construction gives an
 isomorphism $H_0(\Phi_{I,*})\to\K$.  We leave it to the reader to
 check that this sends $i_{\{j\}}(1)$ to $1$ for all $j$.
\end{proof}

\section{Functorality of $\Phi_I$}

\begin{definition}\label{defn-int-fibre}
 Let $\sg\:I\to J$ be a surjective map.  As in
 Section~\ref{sec-de-rham} this gives maps $\sg^*\:P_J\to P_I$ and
 $\sg^*\:W_J\to W_I$ and $\sg^*\:\Om^*_J\to\Om^*_I$.  Next, for any
 map $\nu\:I\to\Z$ we define $\sg_*\nu\:J\to\Z$ by
 $(\sg_*\nu)(j)=\sum_{\sg(i)=j}\nu(i)$.  We then define a map
 $\sg_*\:\tP_I\to\tP_J$ (of abelian groups, not of rings) by
 $\sg_*(t^{[\nu]})=t^{[\sg_*(\nu+1)-1]}$.  We also let
 $\sg_*\:\Lm^*(W_I^\vee)\to\Lm^*(W_J^\vee)$ be dual to the map
 $\sg^*\:\Lm^*(W_J)\to\Lm^*(W_I)$, and we again write $\sg_*$ for the
 map 
 \[ \sg_*\ot\sg_*\:\tP_I\ot\Lm^*(W_I)^\vee \to
     \tP_J\ot\Lm^*(W_J)^\vee.
 \] 
\end{definition}
\begin{remark}\label{rem-covariant-functor}
 It is easy to check that in all the contexts mentioned we have
 $(\tau\sg)_*=\tau_*\sg_*$ for any pair of surjective maps
 $I\xra{\sg}J\xra{\tau}K$.  
\end{remark}

\begin{lemma}\label{lem-int-fibre}
 The map $\sg_*\:\tP_I\to\tP_J$ induces a map $\sg_*\:P_I\to P_J$
 satisfying $\int_J\sg_*(f)=\int_If$.
\end{lemma}
\begin{proof}
 Put $r_I=\sum_{i\in I}t_i$, so that $P_I=\tP_I/(1-r_I)\tP_I$ and
 $r_It^{[nu]}=\sum_i(\nu_i+1)t^{[\nu+e_i]}$.  A straightforward
 calculation shows that $\sg_*(r_It^{[\nu]})=r_J\sg_*(t^{[nu]})$, and
 it follows that $\sg_*$ induces a map $P_I\to P_J$.

 For the integral formula, put $n=|I|-1$ and $m=|J|-1$, so
 $\int_It^{[\nu]}=1/(n+|\nu|)!$ and $\int_Jt^{[\mu]}=(m+|\mu|)!$.  It
 will suffice to show that $n+|\nu|=m+|\mu|$, which is again
 straightforward.
\end{proof}
\begin{remark}\label{rem-singleton}
 If we let $\gm\:I\to 1$ be the unique map to a singleton, we find
 that $P_1=\K$ and $\gm_*(f)=\int_If$.  This gives another way to see
 that $\int_J\sg_*(f)=\int_I f$.
\end{remark}

\begin{lemma}\label{lem-star-adjoint}
 More generally, for $f\in P_I$ and $g\in P_J$ we have
 $\int_If.\sg^*(g)=\int_J\sg_*(f).g$. 
\end{lemma}
\begin{remark}\label{rem-int-fibres}
 One can deduce that in the case $\K=\R$, the map $\sg_*$ is given by
 integrating over fibres of the map $\sg_*\:\Dl_I\to\Dl_J$ of
 simplices.
\end{remark}
\begin{proof}
 We may assume that $f=t^{[\nu]}$ and $g=t^{[\mu]}$ for some
 $\nu\:I\to\N$ and $\mu\:J\to\N$.  Put $n=|I|-1$ and $m=|J|-1$ and
 $\ep=1/(n+|\nu|+|\mu|)!$.  Put $\ov{\nu}=\sg_*(\nu+1)-1$, so that
 $\sg_*(t^{[\nu]})=t^{[\ov{\nu}]}$ and $|\ov{\nu}|=|\nu|+n-m$ and
 $|\ov{\nu}|+|\mu|+m=|\nu|+|\mu|+n$.  Put
 $u_j=(\ov{\nu}_j,\mu_j)$ so
 $\sg_*(t^{[\nu]})t^{[\mu]}=(\prod_ju_j)t^{[\ov{\nu}+\mu]}$.
 and so
 $\int_J\sg_*(t^{[\nu]})t^{[\mu]}=(\prod_ju_j)\ep$.
 
 Next, put $I_j=\sg^{-1}\{j\}$, and let $\Lm_j$ be the set of maps
 $\lm\:I_j\to\N$ with $|\lm|=\mu_j$.  The binomial expansion tells
 us that
 \[ \sg^*t_j^{[\mu_j]} = (\sum_{i\in I_j}t_i)^{[\mu_j]} = 
     \sum_{\lm\in\Lm_j} t^{[\lm]}.
 \]
 Next, for $\lm\in\Lm_j$ put $c_\lm=\prod_{i\in I_j}(\nu_i,\lm_i)$,
 and then put $v_j=\sum_{\lm\in\Lm_j}c_\lm$.  Put 
 \[ \Lm = \prod_j\Lm_j\simeq \{\lm\:I\to\N\st\sg_*(\lm)=\mu\}, \]
 and for $\lm=(\lm_j)_{j\in J}$ put $c_\lm=\prod_jc_{\lm_j}$, and then
 put $v=\sum_{\lm\in\Lm}c_\lm=\prod_jv_j$.  We find that 
 $t^{[\nu]}\sg^*t^{[\mu]}=\sum_{\lm\in\Lm}c_\lm t^{[\nu+\lm]}$.  For these
 terms we have $|\lm|=|\sg_*(\lm)|=|\mu|$ and so
 $\int_It^{[\nu+\lm]}=\ep$.  It follows that
 \[ \int_It^{[\nu]}\sg^*(t^{[\mu]}) = 
     (\sum_\lm c_\lm)\ep = (\prod_jv_j)\ep,
 \]
 so it will suffice to show that $u_j=v_j$. 

 For this, we choose an identification of $I_j$ with the set
 $[d]=\{0,1,\dotsc,d\}$ for some $d\geq 0$.  Let $N_i$ be a totally
 ordered set of size $\nu_i$, and put $N=\coprod_{i\in[d]}N_i$,
 ordered so that $N_i$ comes before $N_{i+1}$.  Now put
 $D=\{i+\half\st 0\leq i<d\}$ and call this the set of ``dividers'';
 we order $N\amalg D$ so that $i+\half$ comes between $N_i$ and
 $N_{i+1}$.  Let $M$ be a totally ordered set of size $\mu_j$, and let
 $U$ be the set of total orderings of $N\amalg D\amalg M$ that are
 compatible with the given orderings of $N\amalg D$ and $M$.  Now
 $|N\amalg D|=\ov{\nu}_j$ and $|M|=\mu_j$ so
 $|U|=(\ov{\nu}_j,\mu_j)=u_j$.  Given an ordering in $U$ we can split
 $M$ along the dividers to get a decomposition
 $M=M_0\amalg\dotsb\amalg M_d$.  Here the sets $M_i$ are consecutive
 intervals, so the decomposition is completely determined by the
 numbers $\lm_i=|M_i|$, which satisfy $\sum_i\lm_i=\mu_i$.  Given the
 decomposition $M=\coprod_iM_i$, the order on $N\amalg D\amalg M$ is
 determined by the relative order of $M_i$ and $N_i$ within
 $N_i\amalg M_i$, for which the number of choices is
 $\prod_i(\nu_i,\lm_i)=c_\lm$.  Using this, one can check that
 $|U|=v_j$, so $u_j=v_j$ as required.    
\end{proof}

\begin{definition}\label{defn-star}
 Let $\sg\:I\to I'$ be an arbitrary map of finite sets.  Given a
 subset $J\sse I$ and an element $\al\in\Tht_{J,*}$ we can interpret
 $\sg$ as a surjection $J\to\sg(J)$ and thus get an element
 $i_{\sg(J)}(\sg_*(\al))\in\Phi_{I',*}$.  We define a map
 $\sg_*\:\Phi_{I,*}\to\Phi'_{I,*}$ by
 $\sg_*(i_J(\al))=i_{\sg(J)}(\sg_*(\al))$. 
\end{definition}

\begin{remark}\label{rem-dl-tht}
 Let $\dl_j\:[n-1]\to[n]$ be the unique increasing map with image
 $[n]\sm\{j\}$.  We can now rewrite Lemma~\ref{lem-dl-tht} as 
 \[ \dl''(\tht_{[n]}) = \dl(\tht_{[n]}) =
     -\sum_{j\in [n]} (-1)^j(\dl_j)_*(\tht_{[n-1]}).
 \]
\end{remark}

\begin{proposition}\label{prop-star-adjoint}
 For $\al\in\Phi_{I,m}$ and $\om\in\Om_{I'}^m$ we have
 $\iip{\sg_*(\al),\om}_{I'}=\iip{\al,\sg^*(\om)}_I$.
\end{proposition}
\begin{proof}
 We may assume that $\al=i_J(f\al_0)$ for some $J\sse I$ and some
 $f\in P_J$ and $\al_0\in\Lm^m(W^\vee_J)$.  Similarly, we may assume
 that $\om=g\om_0$ for some $g\in P_{I'}$ and $\om_0\in\Lm^m(W_{I'})$.
 Put $J'=\sg(J)$ and let $\sg'$ denote the surjective map $\sg\:I'\to
 J'$.  Put $f'=\sg'_*(f)\in P_{J'}$ and
 $\al'_0=\sg'_*(\al_0)\in\Lm^m(W_{J'})^\vee$.  Let $i\:J'\to I'$ be
 the inclusion, so that $i\sg'=\sg$.  Put $g'=i^*g$ and
 $\om'_0=i^*\om$.  From the definitions we then have
 \[ \iip{\sg_*(\al),\om}_{I'} = 
     \int_{J'} \ip{f'\al'_0,g'\om'_0} = 
      \ip{\al'_0,\om'_0} \int_{J'} f'g'.
 \]
 It is elementary that
 \[ \ip{\al'_0,\om'_0} = 
     \ip{\sg'_*(\al_0),i^*(\om_0)} = 
      \ip{\al_0,(\sg')^*i^*\om_0} = 
       \ip{\al_0,\sg^*\om_0}.
 \]
 Similarly, we see from Lemma~\ref{lem-star-adjoint} that
 \[ \int_{J'} f'g' = 
     \int_{J'} \sg'_*(f) i^*(g) = 
      \int_{I'} f\,.\,(\sg')^*i^*(g) = 
       \int_{I'} f\,\sg^*(g).
 \]
 The claim follows directly from this.
\end{proof}

\begin{corollary}\label{cor-star-chain}
 The map $\sg_*\:\Phi_{I,*}\to\Phi_{I',*}$ is a chain map and a
 quasiisomorphism. 
\end{corollary}
\begin{proof}
 We can now identify the above map as a restriction of the map
 $\sg_*\:\hPhi_{I,*}\to\hPhi_{I',*}$, which is dual to the chain map
 $\sg^*\:\Om^*_{I'}\to\Om^*_I$ and so is itself a chain map.  It
 follows from Proposition~\ref{prop-H-Phi} that $\sg_*$ is also a
 quasiisomorphism. 
\end{proof}

\section{De Rham chains on a simplicial set}

We are now in a position to implement Definition~\ref{defn-PhiX}: a
simplicial set $X$ gives a functor $\DDl^\opp\tm\DDl\to\Ch$ by
$(n,m)\mapsto\Z[X_n]\ot\Phi_{[m],*}$, and we write $\Phi_*(X)$ for the
coend.  Thus $\Phi$ is a functor from simplicial sets to chain
complexes that preserves all colimits, and
$\Phi_*(\Dl_n)=\Phi_{[n],*}$, and these properties characterise
$\Phi_*(X)$.  An generator of $\Phi_d(X)$ can be written as $x\ot\al$
for some $x\in X_m$ and $\al\in\Phi_{[m],d}$, subject to the relations
that $x\ot\al$ is a $\K$-linear function of $\al$ and
$\rho^*(x)\ot\al=x\ot\rho_*(\al)$ for all $\rho\:[n]\to[m]$ and
$\al\in\Phi_{[n],d}$.  The differential is just
$\dl(x\ot\al)=x\ot\dl(\al)$. 

Recall that $\Om^d(X)$ is the set of maps $X_n\to\Om^d_{[n]}$ that are
natural for $[n]\in\DDl$.  There is a natural pairing 
\[ \iip{\cdot,\cdot}_X \: \Phi_d(X) \ot \Om^d(X) \to \K \]
given by $\iip{x\ot\al,\om}_X=\iip{\al,\om(x)}_{[m]}$ (for $x\in X_m$ and
$\al\in\Phi_{[m],d}$ and $\om\in\Om^d_X$).

\begin{definition}
 We write $\hPhi_*(X)=\Hom_\K(\Om^*(X),\K)$, so the above pairing
 gives a natural chain map $\xi\:\Phi_*(X)\to\hPhi_*(X)$.
\end{definition}

\begin{remark}\label{rem-Delta-rigid}
 In the rest of this paper, we will have a number of constructions
 related to $\Phi_{I,*}$ that depend on having a total order on $I$.
 If $I$ is totally ordered and $|I|=n+1$ then there is a unique
 order-preserving bijection between $I$ and $[n]=\{0,\dotsc,n\}$.
 Because of this, we can work with the sets $[n]$ where convenient,
 and we will transfer the results to all other finite ordered sets
 without explicit comment.
\end{remark}

We next compare $\Phi_*(X)$ with the usual normalised chain complex
$N_*(X)$.  (We recall the definition: an $n$-simplex $x\in X_n$ is
called \emph{degenerate} if it can be written as $\al^*y$ for some
$y\in X_m$ and some non-injective map $\al\in\DDl([n],[m])$, and
$N_n(X)$ is freely generated over $\K$ by the $n$-simplices modulo the
degenerate ones.)

\begin{proposition}\label{prop-phi-chain}
 There is a natural chain map $\phi\:N_*(X)\to\Phi_*(X)$ given by
 $\phi(x)=(-1)^nx\ot\tht_{[n]}\in\Phi_n(X)$ for all $x\in X_n$.
 (Here $\tht_{[n]}$ is as in Definition~\ref{defn-theta}.)
\end{proposition}
\begin{proof}
 The formula $\phi(x)=(-1)^nx\ot\tht_{[n]}$ certainly defines a
 natural map $X_n\to\Phi_n(X)$ of sets, which extends linearly to give
 a map $\phi\:C_n(X)=\K\{X_n\}\to\Phi_n(X)$ of vector spaces.  We make
 $C_*(X)$ into a chain complex using the alternating sum of face maps
 in the usual way.  We claim that $\phi$ is then a chain map.  Indeed,
 we have 
 \[ \phi(d_ix)= \phi((\dl_i)^*x) = 
     (-1)^{n-1}(\dl_i)^*x\ot\tht_{[n-1]}=
      (-1)^{n-1}x\ot(\dl_i)_*\tht_{[n-1]}.
 \]
 By taking alternating sums and using Remark~\ref{rem-dl-tht} we
 obtain 
 \[ \phi(dx) =
     (-1)^n x\ot\left(-\sum_i(-1)^i(\dl_i)_*\tht_{[n-1]}\right)
      = (-1)^n x\ot\dl(\tht_{[n]}) = \dl(\phi(x)).
 \]
 Now suppose that $x$ is degenerate, say $x=\sg^*(y)$ for some
 surjective map $\sg\:[n]\to[m]$ with $m<n$.  Then
 $\phi(x)=\pm\sg^*(x)\ot\tht_{[n]}=\pm x\ot\sg_*(\tht_{[n]})$ and
 $\sg_*(\tht_{[n]})\in\Lm^n(W^\vee_{[m]})=0$ so $\phi(x)=0$.  There is
 thus an induced chain map $\phi\:N_*(X)\to\Phi_*(X)$ as claimed.
\end{proof}

\begin{proposition}\label{prop-Phi-split}
 There is a natural isomorphism of graded groups
 \[ \bigoplus_m N_m(X) \ot \Tht_{[m],d} \to \Phi_d(X). \]
 (The interaction with differentials is complicated and will not be
 made explicit.)
\end{proposition}
\begin{proof}
 Let $\EE$ be the subcategory of $\DDl$ which contains all the objects
 but only the surjective morphisms, and let $i\:\EE\to\DDl$ be the
 inclusion.  We find that $\Tht$ can be regarded as a functor from
 $\EE$ to the category $\CV_*$ of graded vector spaces over $\K$, and
 if we ignore the differential then $\Phi$ is just the left Kan
 extension $\colim_i\Tht$.  Now consider a simplicial set $X$ and an object
 $V_*\in\CV_*$.  We can define a functor $T\:\DDl\to\CV_*$ by
 $T_n=\Map(X_n,V_*)$ and from the universal properties of coends and
 Kan extensions we see that 
 \[ \CV_*(\Phi_*(X),V_*) = [\DDl,\CV_*](\Phi,T) = 
    [\DDl,\CV_*](\colim_i\Tht,T) = [\EE,\CV_*](\Tht,i^*T). 
 \]
 Now let $\ND_n(X)$ be the set of non-degenerate $n$-simplices in $X$.
 There is an evident map $\coprod_m\EE(n,m)\tm\ND_m(X)\to X_n$ sending
 $(\al,x')$ to $\al^*x'$, and it is a standard fact that this is
 bijective.  (The original reference is~\cite[8.3]{eizi:ssc}, and we
 have given a proof as Lemma~\ref{lem-degen-split} for convenience.)
 We therefore have $T_n=\prod_m\Map(\EE(n,m),T'_m)$, where
 $T'_m=\Map(X'_m,V_*)$.  It follows using the Yoneda Lemma that
 \[ [\EE,\CV_*](\Tht,i^*T) = \prod_m\CV_*(\Tht_{[m],*},T'_m) = 
     \prod_m\CV_*(\Z\{X'_m\}\ot\Tht_{[m],*},V_*) = 
     \CV_*(\bigoplus_m N_m(X)\ot\Tht_{[m],*},V_*). 
 \]
 We now see that $\Phi_d(X)$ and $\bigoplus_m N_m(X) \ot \Tht_{[m],d}$
 represent the same functor, so they are isomorphic in a canonical
 way. 
\end{proof}

\begin{proposition}\label{prop-N-Phi}
 The map $\phi_X\:N_*(X)\to\Phi_*(X)$ is a quasiisomorphism.
\end{proposition}

\begin{remark}
 The case where $X$ is a point is easy.  One way to prove the general
 case would be to show that the functor $H_*\Phi_*(X)$ is homotopy
 invariant, has Mayer-Vietoris sequences, and preserves filtered
 colimits; then the claim would reduce to the usual uniqueness
 argument for homology theories.  Our proof will be slightly
 different; we will rearrange the uniqueness proof so as not to rely
 on homotopy invariance, which instead we deduce as a byproduct.
\end{remark}

\begin{proof}
 Put $\CX=\{X\st \phi_X \text{ is a quasiisomorphism }\}$; we must
 show that this contains all simplicial sets.  It is easy to see that
 $\CX$ is closed under coproducts and filtered colimits.
 Proposition~\ref{prop-H-Phi} tells us that $\Dl_n\in\CX$ for all
 $n$.  Now let $Z$ be an $n$-dimensional simplicial set, and suppose
 inductively that all $(n-1)$-dimensional simplicial sets lie in
 $\CX$.  Let $Y$ be the $(n-1)$-skeleton of $Z$, so we have a pushout
 square of the form 
 \[ \xymatrix{
     A\tm\partial\Dl_n \ar@{ >->}[r]^i \dto_f & A\tm\Dl_n \dto^g \\
     Y \ar@{ >->}[r]_j & Z
    }
 \]
 for some set $A$.  This in turn gives a diagram
 \[ \xymatrix{
     N_*(A\tm\partial\Dl_n) \rto \dto_\phi &
     N_*(A\tm\Dl_n)\op N_*(Y) \rto \dto_{\phi\op\phi} &
     N_*(Z) \dto^\phi \\
     \Phi_*(A\tm\partial\Dl_n) \rto &
     \Phi_*(A\tm\Dl_n)\op N_*(Y) \rto &
     \Phi_*(Z).
    }
 \]
 It is standard that the top row is short exact (giving a
 Mayer-Vietoris sequence in ordinary homology).  Using
 Proposition~\ref{prop-Phi-split} we see that $\Phi_n(X)$ can be split
 naturally as a direct sum of functors of the form $N_m(X)$ for
 various $m$, and it follows that the bottom row is also short exact.
 The first two vertical maps are quasiisomorphisms by the induction
 hypothesis and Proposition~\ref{prop-H-Phi}.  It follows that
 $\phi_Z$ must also be a quasiisomorphism, so $Z\in\CX$.  By induction
 on dimension and passage to colimits we see that $\CX$ contains all
 simplicial sets, as required.
\end{proof}

\subsection{Monoidal properties}

We now define natural maps
$\mu_{X,Y}\:\Om^*(X)\ot\Om^*(Y)\to\Om^*(X\tm Y)$ and
$\mu_{X,Y}\:\Phi_*(X)\ot\Phi_*(Y)\to\Phi_*(X\tm Y)$, in several
stages.

The cohomological version is straightforward.
\begin{definition}\label{defn-mult-Om}
 Given $\om\in\Om^d(X)$ and $\up\in\Om^e(Y)$ we define 
 $\om\wedge\up$ to be the composite 
 \[ X_n\tm Y_n\xra{\om\wedge\up}
    \Om^d_{[n]}\tm\Om^e_{[n]} \xra{\text{mult}}
    \Om^{d+e}_{[n]}.
 \]
 This is natural for $n\in\DDl$ and so gives
 $\om\wedge\up\in\Om^{d+e}(X\tm Y)$.  This construction makes
 $\Om$ into a symmetric monoidal functor from simplicial sets to
 cochain complexes.
\end{definition}

For the homological version, we need to use the set $\Sg(n,m)$ of
$(n,m)$-shuffles; see Appendix~\ref{apx-simplicial} for details of our
approach to this, and various other preliminaries about the simplicial
category.

\begin{definition}\label{defn-s-w}
 In the ring $P_{[n]}=\K[t_0,\dotsc,t_n]/(1-\sum_it_i)$ we put
 $s_i=\sum_{j<i}t_j$, so that $s_0=0$ and $s_{n+1}=1$ and
 $P_{[n]}=\K[s_1,\dotsc,s_n]$.  This gives a basis
 $\{ds_1,\dotsc,ds_n\}$ for $W_{[n]}$.  Recall that $\tW^\vee_{[n]}$
 has basis $e_0,\dotsc,e_n$, and that $W_{[n]}^\vee$ is the subspace
 spanned by the differences $e_i-e_j$.  We put $w_i=e_{i-1}-e_i$, and
 observe that $w_1,\dotsc,w_n$ is a basis for $W_{[n]}^\vee$, with
 $\ip{w_i,s_j}=\dl_{ij}$. 
\end{definition}

The following observation is immediate from the definitions.
\begin{lemma}\label{lem-star-s}
 If $\al\:[n]\to[m]$ is surjective then $\al^*(s_i)=s_{\al^\dag(i)}$
 and so $\al^*(ds_i)=ds_{\al^\dag(i)}$. \qed
\end{lemma}

\begin{lemma}\label{lem-shuffle-iso}
 If $(\zt,\xi)\in\Sg(n,m)$, then the resulting maps
 \begin{align*}
  W_{[n]} \oplus W_{[m]} &\xra{(\zt^*,\xi^*)} W_{[n+m]} \\
  \Lm^*(W_{[n]})\ot \Lm^*(W_{[m]}) &\xra{\zt^*\ot\xi^*}
   \Lm^*(W_{[n+m]})\ot \Lm^*(W_{[n+m]}) 
     \xra{\text{mult}} \Lm^*(W_{[n+m]}) \\
  P_{[n]}\ot P_{[m]} &\xra{\zt^*\ot\xi^*} P_{[n+m]}\ot P_{[n+m]} 
     \xra{\text{mult}} P_{[n+m]} \\
  \Om^*_{[n]}\ot \Om^*_{[m]} &\xra{\zt^*\ot\xi^*} \Om^*_{[n+m]}\ot \Om^*_{[n+m]} 
     \xra{\text{mult}} \Om^*_{[n+m]}
 \end{align*}
 are isomorphisms.  (We will write $\mu_{\zt\xi}$ for any of these maps.)
\end{lemma}
\begin{proof}
 The maps
 \[ [n]' \xra{\zt^\dag} [n+m]' \xla{\xi^\dag} [m]' \]
 give a coproduct decomposition by Lemma~\ref{lem-shuffle-coprod}.
 The claim follows using Lemma~\ref{lem-star-s}.
\end{proof}

\begin{definition}\label{defn-bullet}
 Given a nondecreasing surjective map $\sg\:[n]\to[m]$, we define
 $\sg^\bullet\:W^\vee_{[m]}\to W^\vee_{[n]}$ by
 $\sg^\bullet(w_j)=w_{\sg^\dag(i)}$.  We also write $\sg^\bullet$ for
 $\Lm^k(\sg^\bullet)\:\Lm^*(W^\vee_{[m]})\to\Lm^*(W^\vee_{[n]})$ or
 for 
 \[ \sg^*\ot\sg^\bullet\:\Tht_{[m],*}=P_{[m]}\ot\Lm^*(W^\vee_{[m]})
     \to P_{[n]}\ot \Lm^*(W^\vee_{[n]}) = \Tht_{[n],*}.
 \]
\end{definition} 
\begin{remark}\label{rem-bullet}
 One can check directly from the definitions that
 $\ip{\sg^\bullet(\al),\sg^*(\om)}_{[n]}=\ip{\al,\om}_{[m]}$ and
 $\sg^*(u)\im\sg^\bullet(\al)=\sg^\bullet(u\im\al)$.
\end{remark}

\begin{definition}
 Given a shuffle $(\zt,\xi)\:[n+m]\to[n]\tm[m]$ we define an
 isomorphism 
 \[ \mu_{\zt\xi} \: \Lm^*(W_{[n]}^\vee)\ot\Lm^*(W_{[m]}^\vee) \to
     \Lm^*(W_{[n+m]}^\vee)
 \]
 by $\mu_{\zt\xi}(\al\ot\bt)=\zt^\bullet(\al)\wedge\xi^\bullet(\bt)$.
 We also extend this to an isomorphism
 $\Tht_{[n],*}\ot\Tht_{[m],*}\to\Tht_{[n+m],*}$ by putting
 \[ \mu_{\zt\xi}(f\al_0\ot g\bt_0)=
     \zt^*(f)\xi^*(g)\zt^\bullet(\al)\wedge\xi^\bullet(\bt).
 \]
\end{definition}

\begin{lemma}\label{lem-mu-Theta}
 For all $\al\in\Tht_{[n],d}$ and $\bt\in\Tht_{[m],e}$ and
 $\om\in\Om^d_{[n]}$ and $\up\in\Om^e_{[m]}$ we have
 \[ \ip{\mu_{\zt\xi}(\al\ot\bt),\mu_{\zt\xi}(\om\ot\up)}_{[n+m]} = 
    \ip{\mu_{\zt\xi}(\al\ot\bt),\zt^*(\om)\wedge\xi^*(\up)}_{[n+m]} = 
     (-1)^{|\bt||\om|}\zt^*(\ip{\al,\om}_{[n]})\xi^*(\ip{\bt,\up}_{[m]}).
 \]
 Moreover, the following diagram commutes:
 \[ \xymatrix{
     \Tht_{[n],*}\ot\Tht_{[m],*}
      \drto_{\mu_{\zt\xi}} \rrto^{\tau} & & 
     \Tht_{[m],*}\ot\Tht_{[n],*} 
      \dlto^{\mu_{\xi\zt}} \\
     & \Tht_{[n+m],*}
    }
 \]
 (Here $\tau$ is the usual twist map $\tau(a\ot b)=(-1)^{|a||b|}b\ot a$.)
\end{lemma}
\begin{proof}
 Left to the reader.
\end{proof}

\begin{definition}\label{defn-shuffle-sign}
 We let $\sgn(\zt,\xi)\in\{\pm 1\}$ be the number such that
 \[ \mu_{\zt\xi}(\tht_{[n]}\ot\tht_{[m]})=\sgn(\zt,\xi)\tht_{[n+m]}. \]
\end{definition}

We now recall the standard way to make $N_*$ into a symmetric monoidal
functor (see for example~\cite[Section 29]{ma:soa}).
\begin{definition}\label{defn-shuffle-product}
 We define a map $\mu\:N_n(X)\ot N_m(Y)\to N_{n+m}(X\tm Y)$ (called the
 \emph{shuffle product}) by 
 \[ \mu(x\ot y) =
     \sum_{(\zt,\xi)\in\Sg(n,m)} \sgn(\zt,\xi) (\zt^*(x),\xi^*(y)).
 \]
\end{definition}

There are a number of known generalisations of this construction; for
example, the same formula gives a well-behaved map 
$R_n\ot R_m\to R_{n+m}$ for any simplicial ring $R_\bullet$.  As far
as we understand it, none of these generalisations can be applied
directly to our situation, but nonetheless we can give a definition
along the same lines.

\begin{definition}\label{defn-mult-Phi}
 We define $\mu\:\Phi_*(X)\ot\Phi_*(Y)\to\Phi_*(X\tm Y)$ by 
 \[ \mu((x\ot\al_1)\ot(y\ot\bt_1)) =
     \sum_{(\zt,\xi)\in\Sg(n,m)}
      (\zt^*x,\xi^*y)\ot\mu_{\zt,\xi}(\al_1\ot \bt_1)
 \]
 for $x\in X_n$, $y\in Y_m$, $\al_1\in\Tht_{[n],*}$ and
 $\bt_1\in\Tht_{[m],*}$.  To see that this is well-defined and has good
 properties, we repeat the definition in a more long-winded form as
 follows.  We note that a shuffle $(\zt,\xi)$ gives a nondegenerate
 $(n+m)$-simplex $x_{\zt\xi}\in(\Dl_n\tm\Dl_m)_{n+m}$, and thus a
 basis element in $N_{n+m}(\Dl_n\tm\Dl_m)$.  We then define
 \[ \mu\:\Tht_{[n],*}\ot\Tht_{[m],*} \to
     \Phi_*(\Dl_n\ot\Dl_m) =
      \bigoplus_d N_d(\Dl_n\tm\Dl_m)\ot\Tht_{[d],*}
 \]
 by  
 \[ \mu(\al_1\ot\bt_1) =
     \sum_{\zt,\xi}x_{\zt\xi}\ot\mu_{\zt\xi}(\al_1\ot\bt_1).
 \]
 By
 a slight change of notation, if $J$ and $K$ are any finite, nonempty,
 totally ordered sets we get natural maps
 $\mu\:\Tht_{J,*}\ot\Tht_{K,*}\to\Phi_*(\Dl_J\tm\Dl_K)$.  If
 $J\sse[n]$ and $K\sse[m]$ then $\Dl_J\tm\Dl_K\sse\Dl_n\tm\Dl_m$, so
 we get a map $\mu\:\Tht_{J,*}\ot\Tht_{K,*}\to\Phi_*(\Dl_n\tm\Dl_m)$.
 Adding these up over all $J$ and $K$, we get a map
 $\mu\:\Phi_{[n],*}\ot\Phi_{[m],*}\to\Phi_*(\Dl_n\tm\Dl_m)$, which is
 a natural transformation of functors $\DDl\tm\DDl\to\Ch$.  Given
 simplicial sets $X$ and $Y$ we have functors
 $(\DDl\tm\DDl)^\opp\tm\DDl\tm\DDl\to\CV_*$ given by 
 \[ (p,q,n,m) \mapsto \Z\{X_p\tm X_q\}\ot\Phi_{[n],*}\ot\Phi_{[m],*}
 \]
 and
 \[ (p,q,n,m) \mapsto \Z\{X_p\tm X_q\}\ot\Phi_*(\Dl_n\tm\Dl_m). \]
 The coend of the first is $\Phi_*(X)\ot\Phi_*(Y)$, whereas the coend
 of the second is $\Phi_*(X\tm Y)$.  The maps $\mu$ therefore induce a
 well-defined map $\Phi_*(X)\ot\Phi_*(Y)\to\Phi_*(X\tm Y)$.
\end{definition}

\begin{proposition}
 The maps $\mu\:\Phi_*(X)\ot\Phi_*(Y)\to\Phi_*(X\tm Y)$ make $\Phi$ a
 symmetric monoidal functor from simplicial sets to graded vector
 spaces. 
\end{proposition}

We would also like to know that $\mu$ is a chain map, but the proof of
that fact is long so we will do it separately in
Proposition~\ref{prop-mu-chain}. 

\begin{proof}
 First, for any $(m,n,p)$-shuffle $(\zt,\xi,\tht)$ we can define 
 \[ \mu_{\zt\xi\tht}\:
     \Tht_{[m],*}\ot\Tht_{[n],*}\ot\Tht_{[p],*} \to 
      \Tht_{[m+n+p],*}
 \]
 by the evident analogue of Lemma~\ref{lem-shuffle-iso}.  Using this,
 we define 
 \[ \mu_3\:\Phi_*(X)\ot\Phi_*(Y)\ot\Phi_*(Z)\to\Phi_*(X\tm Y\tm Z) \]
 by 
 \[ \mu_3(x\ot\al_1\ot y\ot\bt_1\ot z\ot\gm_1) = 
     \sum_{\zt,\xi,\tht}
      (\zt^*(x),\xi^*(y),\tht^*(z))\ot
       \mu_{\zt\xi\tht}(\al_1\ot\bt_1\ot\gm_1).
 \]
 Using Lemma~\ref{lem-operad} we see that 
 \[ \mu\circ(\mu\ot 1) = \mu_3 = \mu\circ(1\ot\mu) \:
     \Phi_*(X)\ot\Phi_*(Y)\ot\Phi_*(Z)\to\Phi_*(X\tm Y\tm Z),
 \]
 so we have made $\Phi_*$ into a monoidal functor.  It follows from
 the diagram in Lemma~\ref{lem-mu-Theta} that $\mu$ is also compatible
 with the relevant twist maps, so $\Phi_*$ is a symmetric monoidal
 functor.  
\end{proof}

\begin{proposition}\label{prop-mu-adjoint}
 The maps $\mu\:\Phi_*(X)\ot\Phi_*(Y)\to\Phi_*(X\tm Y)$ and the maps
 $\wedge\:\Om^*(X)\ot\Om^*(Y)\to\Om^*(X\tm Y)$ satisfy 
 \[ \iip{\mu(\al\ot\bt),\om\wedge\up} = 
    (-1)^{|\bt||\om|}\iip{\al,\om}\iip{\bt,\up}.
 \]
\end{proposition}
\begin{proof}
 We may assume that $\al=x\ot\al_1$ and $\bt=y\ot\bt_1$ for some
 $x\in X_n$, $y\in Y_m$, $\al_1\in\Tht_{[n],d}$ and
 $\bt_1\in\Tht_{[m],e}$.  For a nonzero result we must then have
 $\om\in\Om^d(X)$ and $\up\in\Om^e(Y)$, so we can put
 $\om_1=\om(x)\in\Om^d_{[n]}$ and $\up_1=\up(y)\in\Om^e_{[m]}$.  We
 then put $f=\ip{\al_1,\om_1}\in P_{[n]}$ and
 $g=\ip{\bt_1,\up_1}\in P_{[m]}$, so that
 $\iip{\al,\om}=\int_{[n]}f$ and $\iip{\bt,\up}=\int_{[m]}g$.

 Using Lemma~\ref{lem-int-prod} we see that
 \[ \iip{\al,\om}\iip{\bt,\up} 
     = \int_{[n]} f \cdot \int_{[m]} g \\
     = \sum_{(\zt,\xi)\in\Sg(n,m)} \int_{[n+m]} \zt^*(f)\xi^*(g).
 \]
 On the other hand, we have 
 \[ \mu(\al\ot\bt) =
     \sum_{(\zt,\xi)}
      (\zt^*(x),\xi^*(y))\ot\mu_{\zt\xi}(\al_1\ot\bt_1).
 \]
 Here
 \begin{align*}
  \ip{(\zt^*(x),\xi^*(y))\ot\mu_{\zt\xi}(\al_1\ot\bt_1),
      \om\wedge\up}
   &= \ip{\mu_{\zt\xi}(\al_1\ot\bt_1),
          \om(\zt^*(x))\wedge\up(\xi^*(y))} \\
   &= \ip{\mu_{\zt\xi}(\al_1\ot\bt_1),
          \zt^*\om_1\wedge\xi^*\up_1} \\
   &= (-1)^{|\bt||\om|}\zt^*(\ip{\al_1,\om_1})\xi^*(\ip{\bt_1,\up_1}) 
    = (-1)^{|\bt||\om|}\zt^*(f)\,\xi^*(g).
 \end{align*}
 The claim follows.
\end{proof}

\begin{proposition}
 The square
 \[ \xymatrix{
      N_n(X)\ot N_m(Y) \rto^\mu \dto_{\phi\ot\phi} &
      N_{n+m}(X\tm Y) \dto^\phi \\
      \Phi_n(X)\ot\Phi_m(Y) \rto_\mu &
      \Phi_{n+m}(X\tm Y)
    }
 \]
 is commutative.
\end{proposition}
\begin{proof}
 Suppose we have $x\in X_n$ and $y\in Y_m$.  Then 
 \begin{align*}
  \mu(x\ot y) &= 
    \sum_{\zt,\xi}\sgn(\zt,\xi)(\zt^*(x),\xi^*(y)) \\
  \phi\mu(x\ot y) &= 
    (-1)^{n+m} \sum_{\zt,\xi}
     \sgn(\zt,\xi)(\zt^*(x),\xi^*(y))\ot\tht_{[n+m]} \\
  &= (-1)^{n+m} \sum_{\zt,\xi}
     (\zt^*(x),\xi^*(y))\ot\mu_{\zt\xi}(\tht_{[n]}\ot\tht_{[m]}) \\
  &= (-1)^{n+m}\mu((x\ot\tht_{[n]})\ot(y\ot\tht_{[m]})) \\
  &= \mu(\phi\ot\phi)(x\ot y).
 \end{align*}
\end{proof}

\begin{proposition}\label{prop-mu-chain}
 The map $\mu\:\Phi_*(X)\ot\Phi_*(Y)\to\Phi_*(X\tm Y)$ is a chain
 map. 
\end{proposition}

The proof will follow after a number of preparatory results.

Recall that $\dl$ was defined in Definition~\ref{defn-delta} as the
sum of two operators $\dl'$ and $\dl''$.
\begin{lemma}\label{lem-delta-prime}
 For $\al=f\al_0\in\Tht_{[n],*}\leq\Phi_{[n],*}=\Phi_*(\Dl_n)$ and
 $\bt=g\bt_0\in\Tht_{[m],*}\leq\Phi_{[m],*}=\Phi_*(\Dl_m)$ we
 have 
 \[ \dl'(\mu(\al\ot\bt))=\mu(\dl'(\al)\ot\bt +
                          (-1)^{|\al|}\mu(\al\ot\dl'(\bt))
                           \in\Phi_*(\Dl_n\tm\Dl_m).
 \]
\end{lemma}
\begin{proof}
 Let $(\zt,\xi)$ be a shuffle.  Using Remark~\ref{rem-bullet} we see
 that 
 \begin{align*}
  \dl'\mu_{\zt\xi}(\al\ot\bt) 
   &= \dl'(\zt^*(f)\zt^\bullet(\al_0)\wedge\xi^*(g)\xi^\bullet(\bt_0)) \\
   &= -d(\zt^*(f)\xi^*(g))\im
         (\zt^\bullet(\al_0)\wedge\xi^\bullet(\bt_0)) \\
   &= -(\xi^*(g)\zt^*(df)+\zt^*(f)\xi^*(dg))\im
         (\zt^\bullet(\al_0)\wedge\xi^\bullet(\bt_0)) \\
   &= -\zt^\bullet(df\im\al_0)\xi^\bullet(\bt) 
      -(-1)^{|\al|}\zt^\bullet(\al)\wedge\xi^\bullet(dg\im\bt_0) \\
   &= \mu_{\zt\xi}(\dl'(\al)\ot\bt +
       (-1)^{|\al|}\mu_{\zt\xi}(\al\ot\dl'(\bt)).
 \end{align*}
 Taking the sum over all shuffles $(\zt,\xi)$ gives the claimed
 result. 
\end{proof}

We now start to consider the $\dl''$ terms.  

Consider an element $k\in[n+m]$ and an injective map
$(\zt,\xi)\:[n+m]\sm\{k\}\to[n]\tm[m]$.  We say that this pair is
\emph{extendable} if there exists a shuffle 
$(\phi,\psi)\:[n+m]\to[n]\tm[m]$ extending $(\zt,\xi)$.  We will need
to classify the possible extensions.  We first suppose that 
$0<k<n+m$.  In that case, extendability means precisely that one of
the following three things must hold.
\begin{itemize}
 \item[(0)] For some $(i,j)\in[n]'\tm[m]'$ we have
  $(\zt,\xi)(k-1)=(i-1,j-1)$ and $(\zt,\xi)(k+1)=(i,j)$.  Here we say
  that $(\zt,\xi)$ has a \emph{diagonal gap}.  There are two possible
  extensions, given by $(\phi,\psi)(k)=(i-1,j)$ and
  $(\phi,\psi)(k)=(i,j-1)$.
 \item[(1)] For some $(i,j)\in\{1,\dotsc,n-1\}\tm[m]$ we have
  $(\zt,\xi)(k-1)=(i-1,j)$ and $(\zt,\xi)(k+1)=(i+1,j)$.  Here we say
  that $(\zt,\xi)$ has a \emph{horizontal gap}.  There is only one
  possible extension, given by $(\phi,\psi)(k)=(i,j)$.
 \item[(2)] For some $(i,j)\in [n]\tm\{1,\dotsc,m-1\}$ we have
  $(\zt,\xi)(k-1)=(i,j-1)$ and $(\zt,\xi)(k+1)=(i,j+1)$.  Here we say
  that $(\zt,\xi)$ has a \emph{vertical gap}.  There is only one
  possible extension, given by $(\phi,\psi)(k)=(i,j)$.
\end{itemize}
The situation is similar if $k=0$, but with some slight adjustments.
We must have either $(\zt,\xi)(1)=(1,0)$ or $(\zt,\xi)(1)=(0,1)$
(otherwise there is not room for $(\zt,\xi)$ to be injective).  In
these cases we say that $(\zt,\xi)$ has a horizontal (resp. vertical)
gap.  Either way, there is a unique extension, with
$(\phi,\psi)(0)=(0,0)$.  Similarly, if $k=n+m$ then we can have only a
horizontal or vertical gap, and there is a unique extension given by
$(\phi,\psi)(n+m)=(n,m)$.  

(This division into three cases is the same as in the well-known proof
that the product in Definition~\ref{defn-shuffle-product} is a chain
map.)

Given an extendable pair $(\zt,\xi)$ and an extension $(\phi,\psi)$, 
the expression $\mu(f\al_0\ot g\bt_0)\in\Phi_*(\Dl_n\tm\Dl_n)$
contains a term $(\phi,\psi)\ot\mu_{\phi\psi}(f\al_0\ot g\bt_0)$, 
so $\dl''\mu(f\al_0\ot g\bt_0)$ contains a term
$-(\zt,\xi)\ot\rho_{\phi\psi}$, where
\[ \rho_{\phi\psi} =
    \res^{[n+m]}_{[n+m]\sm\{k\}}(\phi^*(f)\psi^*(g))
     (dt_k\im\mu_{\phi\psi}(\al_0\ot\bt_0)) =
      \zt^*(f)\xi^*(g)
      (dt_k\im\mu_{\phi\psi}(\al_0\ot\bt_0)).
\]

\begin{lemma}\label{lem-diagonal-gap}
 Suppose that $(\zt,\xi)\:[n+m]\sm\{k\}\to[n]\tm [m]$ has a diagonal
 gap between $(i-1,j-1)$ and $(i,j)$, and let $(\phi,\psi)$ and
 $(\ov{\phi},\ov{\psi})$ be the two shuffles that extend $(\zt,\xi)$.
 Then for any $\al_0\in\Lm^*(W_{[n]}^\vee$ and
 $\bt_0\in W_{[m]}^\vee$ we have
 \[ dt_k\im\mu_{\phi\psi}(\al_0\ot\bt_0) \;\;+\;\; 
    dt_k\im\mu_{\ov{\phi}\ov{\psi}}(\al_0\ot\bt_0) = 0.
 \] 
\end{lemma}
\begin{proof}
 Write $\al_0$ as $\al_1+w_i\wedge\al_2$, where $\al_1$ and $\al_2$
 involve only the generators $w_p$ with $p\neq i$.  In particular,
 this means that $dt_i\im\al_0=-\al_2$.  Write $\bt_0$ as
 $\bt_1+w_j\wedge\bt_2$ in the same way.  

 As there is a diagonal gap, we must have $0<k<n+m$.  We have the
 following table of values: 
 \[ \begin{array}{|c|c|c|c|c|}
     \hline
         & \phi & \ov{\phi} & \psi & \ov{\psi} \\ \hline
     k-1 &  i-1 &  i-1      & j-1  &  j-1      \\ \hline
     k   &  i   &  i-1      & j-1  &  j        \\ \hline
     k+1 &  i   &  i        & j    &  j        \\ \hline
    \end{array}
 \]
 and using this we see that $k=\phi^\dag(i)=\ov{\psi}^\dag(j)$ and
 $k+1=\ov{\phi}^\dag(i)=\psi^\dag(j)$.  On the other hand, for all
 $p\neq i$ we have $\phi^\dag(p)=\ov{\phi}^\dag{p}\not\in\{k,k+1\}$,
 and for all $q\neq j$ we have
 $\psi^\dag(j)=\ov{\psi}^\dag(j)\not\in\{k,k+1\}$, so
 $\phi^\bullet(\al_1)=\ov{\phi}^\bullet(\al_1)$ and
 $\phi^\bullet(\al_2)=\ov{\phi}^\bullet(\al_2)$.  Similarly, 
 $\psi^\bullet(\bt_1)=\ov{\psi}^\bullet(\bt_1)$ and
 $\psi^\bullet(\bt_2)=\ov{\psi}^\bullet(\bt_2)$. Put 
 \[ \nu =
    \mu_{\phi\psi}(\al_0\ot\bt_0) +
    \mu_{\ov{\phi}\ov{\psi}}(\al_0\ot\bt_0). 
 \]
 We see that
 \begin{align*}
  \mu_{\phi\psi}(\al_0\ot\bt_0) &= 
   (\phi^\bullet(\al_1)+w_k\wedge\phi^\bullet(\al_2))\wedge
    (\psi^\bullet(\bt_1)+w_{k+1}\wedge\psi^\bullet(\bt_2)) \\
  \mu_{\ov{\phi}\ov{\psi}}(\al_0\ot\bt_0) &= 
   (\phi^\bullet(\al_1)+w_{k+1}\wedge\phi^\bullet(\al_2))\wedge
    (\psi^\bullet(\bt_1)+w_k\wedge\psi^\bullet(\bt_2)) \\
  \nu &= 2\phi^\bullet(\al_1)\wedge\psi^\bullet(\bt_1) + 
   (-1)^{|\al|}(w_k+w_{k+1})\wedge\phi^\bullet(\al_1)\wedge\psi^\bullet(\bt_2) + 
   (w_k+w_{k+1}) \wedge\phi^\bullet(\al_2)\wedge\psi^\bullet(\bt_1) \\
  s_k\im\nu = s_{k+1}\im\nu &= 
   (-1)^{|\al|+1}\phi^\bullet(\al_1)\wedge\psi^\bullet(\bt_2) -
   \phi^\bullet(\al_2)\wedge\psi^\bullet(\bt_1) \\
  t_k\im\nu &= (s_{k+1}-s_k)\im\nu = 0.
 \end{align*}
\end{proof}
\begin{corollary}
 With $\phi,\psi,\ov{\phi}$ and $\ov{\psi}$ as in
 Lemma~\ref{lem-diagonal-gap}, we have
 $\rho_{\phi\psi}+\rho_{\psi\phi}=0$. 
\end{corollary}
\begin{proof}
 This follows from the expression
 \[ \rho_{\phi\psi} =
      \zt^*(f)\xi^*(g)
      (dt_k\im\mu_{\phi\psi}(\al_0\ot\bt_0)).
 \] 
\end{proof}

We next consider the case of a pair
$(\zt,\xi)\:[n+m]\sm\{k\}\to[n]\tm[m]$ that has a horizontal gap at
$i$, and thus a unique extension $(\phi,\psi)$.  We originally defined
shuffles as maps $[p+q]\to[p]\tm[q]$ with certain properties, but we
can extend the notion in an evident way to cover maps $I\to J\tm K$
where $I$, $J$ and $K$ are any finite, totally ordered sets with
$|I|=|J|+|K|-1$.  In this slightly extended sense, we see that
$(\zt,\xi)\:[n+m]\sm\{k\}\to([n]\sm\{i\})\tm[m]$ is a shuffle, so it
gives a map
\[ \mu_{\zt\xi} \: \Tht_{[n]\sm\{i\},*} \ot \Tht_{[m],*} \to
    \Tht_{[n+m]\sm\{k\}}. 
\]

\begin{lemma}\label{lem-horizontal-gap}
 Suppose we elements $\al=f\al_0\in\Tht_{[n],*}$ and
 $\bt=g\bt_0\in\Tht_{[m],*}$.  Then, in the situation described above
 we have
 \[ \rho_{\phi\psi} =
     \mu_{\zt\xi}(\res^{[n]}_{[n]\sm\{i\}}(f)\,(dt_i\im\al_0)\ot g\bt_0).
 \]
\end{lemma}
\begin{proof}
 We will cover the case where $0<k<n+m$, leaving the adjustments for
 $k=0$ and $k=n+m$ to the reader.  We then have
 $\phi(k-1)=\zt(k-1)=i-1$ and $\phi(k)=i$ and
 $\phi(k+1)=\zt(k+1)=i+1$.  Also, for some $j$ we have
 $\psi(k-1)=\psi(k)=\psi(k+1)=\xi(k-1)=\xi(k+1)=j$.  Using the expression
 \[ \rho_{\phi\psi} =
      \zt^*(f)\xi^*(g)
      (dt_k\im\mu_{\phi\psi}(\al_0\ot\bt_0))
 \] 
 we reduce to the case $f=g=1$, in which case we must prove that 
 \[ dt_k\im \mu_{\phi\psi}(\al_0\ot\bt_0) = 
     \mu_{\zt\xi}((dt_i\im\al_0)\ot\bt_0).
 \]
 
 We write
 \[ \al_0 = \al_1 + w_i\wedge\al_2 + w_{i+1}\wedge\al_3 +
             w_i\wedge w_{i+1}\wedge\al_4,
 \]
 where $\al_1,\dotsc,\al_4$ do not involve $w_i$ or $w_{i+1}$.  Put
 $\ov{\al}_t=\phi^\bullet(\al_t)$ and $\ov{\bt}_0=\psi^\bullet(\bt_0)$.
 Then 
 \[ \mu_{\phi\psi}(\al_0\ot\bt_0) = 
     (\ov{\al}_0 + w_k\wedge\ov{\al}_1 + w_{k+1}\wedge\ov{\al}_2 +
        w_k\wedge w_{k+1} \wedge \ov{\al}_3) \wedge \ov{\bt}_0,
 \]
 and none of the terms $\ov{\al}_t$ or $\ov{\bt}_0$ involves $w_k$ or
 $w_{k+1}$.  Using this together with the relation $t_k=s_{k+1}-s_k$
 we obtain
 \[ dt_k\im \mu_{\phi\psi}(\al_0\ot\bt_0) = 
     (\ov{\al}_2 - \ov{\al}_3 + (w_k+w_{k+1})\wedge\ov{\al}_4)
      \wedge\ov{\bt}_0.
 \]
 We now consider the map $\mu_{\zt\xi}$ arising from the shuffle 
 \[ (\zt,\xi) \: [n+m]\sm\{k\} \to ([n]\sm\{i\}) \tm [m]. \]
 Here the natural basis to use for $W^\vee_{[n]\sm\{i\}}$ is the list 
 \[ e_1-e_0,\dotsc,e_{i-1}-e_{i-2},e_{i+1}-e_{i-1},
    e_{i+2}-e_{i+1}, \dotsc , e_n-e_{n-1},
 \]
 or in other words 
 \[ w_1,\dotsc,w_{i-1},w_i+w_{i+1},w_{i+2},\dotsc,w_n. \]
 Similarly, the natural basis for $W_{[n+m]\sm\{k\}}^\vee$ is 
 \[ w_1,\dotsc,w_{k-1},w_k+w_{k+1},w_{k+2},\dotsc,w_{n+m}. \]
 We see that $\zt^\bullet(w_p)=w_{\zt^\dag(p)}=w_{\phi^\dag(p)}$ for
 $p\neq i+1$ and $\zt^\bullet(w_i+w_{i+1})=w_k+w_{k+1}$.  Also, we
 have
 \begin{align*}
  dt_i\im\al_0 &=
   (ds_{i+1}-ds_i)\im
    (\al_1 + w_i\wedge\al_2 + w_{i+1}\wedge\al_3 +
      w_i\wedge w_{i+1}\wedge\al_4) \\
  &= \al_2 - \al_3 + (w_i+w_{i+1})\wedge\al_4, 
 \end{align*} 
 so 
 \[ \mu_{\zt\xi}((dt_i\im\al_0)\ot\bt_0) = 
     (\ov{\al}_2 - \ov{\al}_3 + (w_k+w_{k+1})\wedge\ov{\al}_4)
      \wedge\ov{\bt}_0, 
 \]
 as required.
\end{proof}

\begin{lemma}\label{lem-vertical-gap}
 If $(\zt,\xi)\:[n+m]\sm\{k\}\to[n]\tm[m]$ has a vertical gap at $j$
 and $(\phi,\psi)$ is the unique extension of $(\zt,\xi)$ then
 \[ \rho_{\phi\psi} = 
     (-1)^{|\al|} \mu_{\zt\xi}(f\al_0,
      \res^{[m]}_{[m]\sm\{j\}}(g)(dt_j\im\bt_0)).
 \]
\end{lemma}
\begin{proof}
 This follows from Lemma~\ref{lem-horizontal-gap} by applying suitable
 twist maps.
\end{proof}

\begin{corollary}
 In $\Phi_*(\Dl_n\tm\Dl_m)$ we have 
 \[ \dl(\mu(\al\ot\bt)) =
     \mu(\dl(\al)\ot\bt + (-1)^{|\al|}\al\ot\dl(\bt)).
 \]
\end{corollary}
\begin{proof}
 Lemma~\ref{lem-delta-prime} tells us that this holds when $\dl$ is
 replaced by $\dl'$, so we need only prove the corresponding formula
 for $\dl''$.  We have seen that $\dl''(\mu(\al\ot\bt))$ is a sum of
 terms $-(\zt,\xi)\ot\rho_{\phi\psi}$, one for each extendable pair
 $(\zt,\xi)$ and each extension $(\phi,\psi)$.  The terms where
 $(\zt,\xi)$ has a diagonal gap all cancel out in pairs, by
 Lemma~\ref{lem-diagonal-gap}.  Those where $(\zt,\xi)$ has a
 horizontal gap add up to give $\mu(\dl'(\al)\ot\bt)$, as we see from
 Lemma~\ref{lem-horizontal-gap}.  The remaining terms give
 $(-1)^{|\al|}\mu(\al\ot\dl(\bt))$, by Lemma~\ref{lem-vertical-gap}.
\end{proof}

\begin{proof}[Proof of Proposition~\ref{prop-mu-chain}]
 The group $\Phi_*(X)\ot\Phi_*(Y)$ is generated by terms of the form
 $(x\ot\al)\ot(y\ot\bt)$ with $x\in X_n$ and $y\in Y_m$ and
 $\al\in\Tht_{[n],*}$ and $\bt\in\Tht_{[m],*}$.  We then have 
 \begin{align*}
  \mu((x\ot\al)\ot(y\ot\bt)) &= (x,y)\ot\mu(\al\ot\bt) \\
  \dl(\mu((x\ot\al)\ot(y\ot\bt))) &= (x,y)\ot\dl(\mu(\al\ot\bt)) \\
   &= (x,y)\ot\mu(\dl(\al)\ot\bt+(-1)^{|\al|}\al\ot\dl(\bt)) \\
   &= \mu((x\ot\dl(\al))\ot(y\ot\bt)) +
       (-1)^{|\al|}\mu((x\ot\al)\ot(y\ot\dl(\bt))).
 \end{align*}
\end{proof}

\section{The colimit description}

In this section, we explain and prove Theorem~\ref{thm-Phi-colim},
which asserts that $\Phi_*(X)$ can be written as a colimit of the
groups $\Hom(\tH_*(S^A),\tN_*(S^A\Smash X_+))$, as $A$ runs over the
category of finite sets and injective maps.

\begin{definition}
 Given a finite set $A$, we put $S^A=\bigSmash_{a\in A}S^1$, where
 $S^1=\Dl_1/\partial\Dl_1$.  More explicitly, we define
 $BA=\prod_{a\in A}\Dl_1$, so that $(BA)_n=\Map(A,\DDl([n],[1]))$.
 We then put 
 \[ (\partial BA)_n = 
     \Map(A,\DDl([n],[1])) \sm 
      \Map(A,\EE([n],[1])), 
 \]
 which defines a subcomplex $\partial BA$.  Finally, we have
 $S^A=BA/\partial BA$.
\end{definition}

It is clear that if $|A|=n$ then $S^A$ is a model of the sphere $S^n$,
so that $\tH_*(S^A)$ is a copy of $\Z$, concentrated in degree $n$.
However, there is no natural choice of generator for this group.
Instead, the best thing to say is that there is a natural isomorphism
$\Lm^n\Z\{A\}\to\tH_n(S^A)$.

\begin{definition}\label{defn-z-alpha}
 Given a set $A$ with $|A|=m$ and a simplex $\al\in(BA)_d$ we define
 \[ z(\al)\in\tH_m(S^A)\ot\Phi_{[d],d-m} = 
             \tH_m(S^A)\ot\Phi_{d-m}(\Dl_d)
 \]
 as follows.  First, we note that $\Map(A,[1])$ can be regarded as a
 partially ordered set using the pointwise order, and
 \[ (BA)_d = \text{Poset}([d],\Map(A,[1]))
     = \prod_{a\in A}\DDl([d],[1]).
 \]
 Thus $\al$ gives a system of maps $\al_a\:[d]\to[1]$. 
 \begin{itemize}
  \item[(a)] If any $\al_a$ is constant (or equivalently, not
   surjective) we put $z(\al)=0$.
  \item[(b)] Otherwise, we define $f\:A\to[d]'$ by
   $f(a)=\al_a^\dag(1)$.  If $f$ is not injective, we again
   put $z(\al)=0$.
  \item[(c)] Otherwise, we put 
   \begin{align*}
    U &= \K\{w_{f(a)}\st a\in A\} \\
    V &= \ker(\al_*\:W^\vee_{[d]}\to W^\vee_{\Map(A,[1])}) \\
      &= \K\{w_i\st\al(i)=\al(i-1)\} 
       = \K\{w_i\st i\not\in f(A)\}.
   \end{align*}
   (Here we are using the notation of Definition~\ref{defn-s-w}.)
   We find that $W^\vee_{[d]}=U\op V$, so there is a natural
   isomorphism
   $\Lm^m(U)\ot\Lm^{d-m}(V)\to\Lm^d(W_{[d]}^\vee)=\K\tht_{[d]}$.
   Moreover, the map $a\mapsto w_{f(a)}$ induces an isomorphism
   $\tH_m(S^A)=\Lm^m\K\{A\}\to\Lm^m(U)$, and there are natural
   inclusions 
   \[ \Lm^{d-m}(V) \leq \Lm^{d-m}(W^\vee_{[d]}) \leq
       \Tht_{[d],d-m} \leq \Phi_{d-m}(\Dl_d).
   \]
   By putting these together, we get a map
   $\Lm^d(W^\vee_{[d]})\to\tH_m(S^A)\ot\Phi_{d-m}(\Dl_d)$.  We write
   $z(\al)$ for the image of $\tht_{[d]}$ under this map. 
 \end{itemize}
\end{definition}

\begin{remark}\label{rem-z-split}
 For some purposes it is useful to be more explicit.  Suppose that we
 are in case~(c) of the definition, so that $f\:A\to[d]'$ is
 injective.  We can then list the elements of $A$ as
 $\{a_1,\dotsc,a_m\}$, ordered in such a way that
 $f(a_1)<\dotsb<f(a_m)$.  Similarly, we list the elements of $[d]'\sm
 f(A)$ as $\{j_1<j_2<\dotsb<j_{d-m}\}$.  There is then a number
 $\ep(\al)\in\{\pm 1\}$ such that
 \[ \tht_{[d]} =
      \ep(\al)\,w_{f(a_1)}\wedge\dotsb\wedge w_{f(a_m)}\wedge
       w_{j_1}\wedge\dotsb\wedge w_{j_{d-m}}.
 \]
 Put
 \begin{align*}
  u(\al) &= w_{f(a_1)}\wedge\dotsb\wedge w_{f(a_m)} \in
    \Lm^m(U) \\
  z'(\al) &= a_1\wedge\dotsb\wedge a_m \in
    \Lm^m(\K\{A\}) = \tH_m(S^{|A|}) \\
  z''(\al) &= w_{j_1}\wedge\dotsb\wedge w_{j_{d-m}} \in
    \Lm^{d-m}(W^\vee_{[d]}).
 \end{align*}
 In this notation, the defining property of $\ep(\al)$ is that
 $\tht_{[d]}=\ep(\al)u(\al)\wedge z''(\al)$.  We find that
 $z(\al)=\ep(\al)\,z'(\al)\ot z''(\al)$.
\end{remark}

\begin{definition}\label{defn-phi}
 For any simplicial set $X$ we define
 $\phi\:C_d(BA\tm X)\to\tH_m(S^A)\ot\Phi_{d-m}(X)$ as follows.  Any
 $d$-simplex in $BA\tm X$ has the form $(\al,x)$ where $x\in X_d$ and
 $\al$ is as in Definition~\ref{defn-z-alpha}.  The simplex $x$
 corresponds to a map $\hx\:\Dl_d\to X$.  We put
  \[ \phi(\al,x)=(1\ot\hx_*)(z(\al)) =
      \ep(\al)z'(\al) \ot (x\ot z''(\al)).
 \]
\end{definition}

\begin{remark}\label{rem-smash-factor}
 Clause~(a) in Definition~\ref{defn-z-alpha} tells us that the map
 $\phi$ factors through $\tC_*(S^A\Smash X_+)$, and similarly
 $\phi^\#$ induces a map 
 \[ \Hom(\tH_*(S^A),\tC_*(S^A\Smash X_+)) \to \Phi_*(X). \]
\end{remark}

\begin{lemma}\label{lem-phi-degen}
 If the simplex $(\al,x)\in(BA\tm X)_d$ is degenerate then
 $\phi(\al,x)=0$.  
\end{lemma}
\begin{proof}
 As $(\al,x)$ is degenerate, there must exist a surjection
 $\sg\:[d]\to[e]$ (with $e<d$) and a map $\bt\:[a]\to\Map(A,[1])$ and
 a simplex $y\in X_e$ such that $\al=\bt\sg$ and $x=\sg^*(y)$.  As
 $e<d$ we must have $\sg(i-1)=\sg(i)$ for some $i>0$.  As
 $\al=\bt\sg$ this means that $\al(i)=\al(i-1)$, so $w_i\in V$.
 Clearly $\sg_*(w_i)=\sg_*(e_i-e_{i-1})=e_{\sg(i)}-e_{\sg(i-1)}=0$, so
 $\sg_*=0$ on $\Lm^{d-m}(V)$, so $(1\ot\sg_*)(z(\al))=0$.  By
 definition we have 
 \[ \phi(\al,x)=(1\ot\hx_*)(z(\al)) = 
     (1\ot\hy_*)(1\ot\sg_*)(z(\al)) = 0.
 \]
\end{proof}

\begin{corollary}
 There are induced maps
 $\phi\:\tN_*(S^A\Smash X_+)\to\tH_*(S^A)\ot\Phi_*(X)$. \qed
\end{corollary}

\begin{definition}
 Put
 \[ U_*(A,X) = \Hom(\tH_*(S^A),\tN_*(S^A\Smash X_+)). \]
 As $\tH_*(S^A)$ is invertible under the tensor product, the map
 $\phi$ gives rise to an adjoint map $U_*(A,X)\to\Phi_*(X)$, which we
 denote by $\phi^\#$.
\end{definition}

Now consider $\al\:[d]\to\Map(A,[1])$ and $i\in[d]$, giving a map
$\al\dl_i\:[d-1]\to\Map(A,[1])$ and an element
$z(\al\dl_i)\in\tH_m(S^{|A|})\ot\Tht_{[d-1],d-m-1}$.  We can regard
$\dl_i$ as a bijection $[d-1]\to[d]\sm\{i\}$, so we get an element
$(1\ot(\dl_i)_*)z(\al\dl_i)\in\tH_m(S^{|A|})\ot\Tht_{[d]\sm\{i\},d-m-1}$. 
We also have a map $\tau_i\:\Tht_{[d],*}\to\Tht_{[d]\sm\{i\},*}$ given
by $\tau_i(\zt)=dt_i\im\zt$.

\begin{lemma}\label{lem-delta-z}
 $(-1)^i(1\ot(\dl_i)_*)z(\al\dl_i)=(-1)^{m+1}(1\ot\tau_i)(z(\al))$.
\end{lemma}
\begin{proof}
 We will consider the case $0<i<d$; small adjustments for the end
 cases are left to the reader.  Note that $\al_a\:[d]\to[1]$ is
 surjective iff ($\al_a(0)=0$ and $\al_a(d)=1$) iff $\al\dl_i$ is
 surjective.  We may assume that this holds for all $a$, otherwise
 both sides of the claimed identity are zero.  Next, put
 $f(a)=\al_a^\dag(1)$ as before, and $g(a)=(\al_a\dl_i)^\dag(1)$.  By
 a check of the various possible cases, we see that 
 \[ g(a) = \sg_i(f(a)) =
     \begin{cases}
      f(a) & \text{ if } f(a)\leq i \\
      f(a)-1 & \text{ if } i < f(a).
     \end{cases}
 \]
 It follows that $g$ is injective unless $\{i,i+1\}\sse f(A)$.  

 Suppose that $\{i,i+1\}\sse f(A)$, so $g$ is not injective, so
 $z(\al\dl_i)=0$.  In this case $z''(\al)$ does not involve $w_i$ or
 $w_{i+1}$, so $dt_i\im z''(\al)=(ds_{i+1}-ds_i)\im z''(\al)=0$, and
 we see that both sides of the claimed identity are again zero.

 Suppose instead that $\{i,i+1\}\not\sse f(A)$.  One checks that
 $z'(\al)=z'(\al\dl_i)$.  Let $w'$ be the wedge of all the factors
 $w_{j_t}$ in $z''(\al)$ with $j_t\in\{i,i+1\}$, and let
 $w''$ be the wedge of the remaining factors, so 
 \[ z''(\al) = \ep' w'\wedge w'' \]
 for some $\ep'\in\{\pm 1\}$.  Because $\{i,i+1\}\not\sse f(A)$ we
 must have $w'=w_i$ or $w'=w_{i+1}$ or $w'=w_i\wedge w_{i+1}$.  In
 computing $(\dl_i)_*z''(\al\dl_i)$, we use the fact that
 $(\dl_i)_*w_j=(\dl_i)_*(e_i-e_{i-1})=w_{\dl_i(j)}$ except in the case
 $j=i$, in which case we have $(\dl_i)_*(w_i)=w_i+w_{i+1}$.  There are
 three cases to consider.
 \begin{itemize}
  \item[(a)] If $w'=w_i$ (so $i\not\in f(A)$ but $i+1\in f(A)$) we find
   that $ds_i\im z''(\al)=-\ep' w''$ and $ds_{i+1}\im z''(\al)=0$ so
   $dt_i\im z''(\al)=\ep'w''$.  On the other
   hand, as $i\not\in f(A)$ we have
   $\dl_i(g(a))=\dl_i(\sg_i(f(a)))=f(a)$ for all $a$, so
   $\dl_i([d-1]'\sm g(A))=([d]'\sm f(A))\sm\{i\}$, so
   $(\dl_i)_*z''(\al\dl_i)=w''$.  We next need to understand
   $\ep(\al\dl_i)$.  By definition we have 
   \[ \ep(\al\dl_i)
        u(\al\dl_i) \wedge z''(\al\dl_i) = 
         \tht_{[d-1]}.
   \]
   As $\dl_if=g$ and $(\dl_i)_*z''(\al\dl_i)=w''$ we see that
   $u(\al)=u(\al\dl_i)$ and
   \[ \ep(\al\dl_i) u(\al) \wedge w'' = \tht_{[d]\sm\{i\}}. \]
   We then multiply both sides on the left by $w_i$ to get 
   \[ (-1)^m\ep(\al\dl_i)
        u(\al) \wedge w_i\wedge w'' = 
         (-1)^{i-1}\tht_{[d]}.
   \]
   On the other hand, by the definitions of $\ep(\al)$ and $\ep'$ we
   have
   \[ \ep'\ep(\al) u(\al) \wedge w_i\wedge w'' = \tht_{[d]}. \]
   It follows that $\ep(\al\dl_i)=(-1)^{m+i+1}\ep'\ep(\al)$.  This
   gives 
   \begin{align*}
    (-1)^i(1\ot(\dl_i)_*)z(\al\dl_i)
     &= (-1)^i\ep(\al\dl_i)z'(\al\dl_i)\ot (\dl_i)_*z''(\al\dl_i) \\
     &= (-1)^i(-1)^{m+i+1}\ep'\ep(\al)z'(\al) w'' \\
     &= (-1)^{m+1}\ep(\al)z'(\al)(dt_i\im z''(\al)) 
      = (-1)^{m+1}(1\ot\tau_i)(z(\al))
   \end{align*}
   as required.
  \item[(b)] Now suppose instead that $w'=w_{i+1}$, so that
   $i\in f(A)$ but $i+1\not\in f(A)$.  We find that
   $ds_i\im z''(\al)=0$ and $s_{i+1}\im z''(\al)=-w''$, so
   $dt_i\im z''(\al)=-w''$.  On the other hand, we find that
   \[ \dl_i(g(a))=\dl_i\sg_i(f(a)) = 
       \begin{cases}
        f(a) & \text{ if } f(a)\neq i \\
        i+1  & \text{ if } f(a)=i.
       \end{cases}
   \]
   From this we see that
   $\dl_i([d-1]'\sm g(A))=([d]'\sm f(A))\sm\{i+1\}$, and thus that
   $(\dl_i)_*z''(\al\dl_i)=w''$.  We next need to understand
   $\ep(\al\dl_i)$.  From the definitions we have
   \[ \ep(\al\dl_i)
        w_{\dl_ig(a_1)}\wedge\dotsb\wedge w_{\dl_ig(a_m)}
         \wedge (\dl_i)_*z''(\al\dl_i) = 
         \tht_{[d]\sm\{i\}}.
   \]
   Let $r$ be such that $f(a_r)=i$, and let $v$ be the wedge of the
   terms $w_{f(a_p)}$ for $p\neq r$.  The above equation can then be
   written as  
   \[ (-1)^r\ep(\al\dl_i) w_{i+1}\wedge v\wedge w'' =
       \tht_{[d]\sm\{i\}}.
   \]
   We now multiply both sides on the left by $w_i$ to get 
   \[ (-1)^{r}\ep(\al\dl_i) w_i\wedge w_{i+1} \wedge v \wedge w'' = 
       (-1)^{i+1}\tht_{[d]}
   \]
   On the other hand, by the definitions of $\ep(\al)$ and $\ep'$ we
   have
   \[ (-1)^r\ep'\ep(\al)
        w_i\wedge v \wedge w_{i+1}\wedge w'' = 
         \tht_{[d]}.
   \]
   It follows that $\ep(\al\dl_i)=(-1)^{m+i}\ep'\ep(\al)$.  This gives 
   \begin{align*}
    (-1)^i(1\ot(\dl_i)_*)z(\al\dl_i)
     &= (-1)^i\ep(\al\dl_i)z'(\al\dl_i)\ot (\dl_i)_*z''(\al\dl_i) \\
     &= (-1)^i(-1)^{m+i}\ep'\ep(\al)z'(\al) w'' \\
     &= (-1)^{m}\ep(\al)z'(\al)(-dt_i\im z''(\al)) 
      = (-1)^{m+1}(1\ot\tau_i)(z(\al))
   \end{align*}
   as required.
  \item[(c)] Finally, suppose that neither $i$ nor $i+1$ is in $f(A)$,
   so $w'=w_i\wedge w_{i+1}$.  As this has even degree we have
   $\ep'=1$ and $z''(\al)=w'\wedge w''$.  We then have
   $ds_i\im z''(\al)=-w_{i+1}\wedge w''$ and
   $ds_{i+1}\im z''(\al)=w_i\wedge w''$ so
   $dt_i\im z''(\al)=(w_i+w_{i+1})\wedge w''$.  On the other hand,
   as in case~(a) we see that $f=\dl_ig$ and
   $\dl_i([d-1]'\sm g(A))=([d]'\sm f(A))\sm\{i\}$.  Suppose that $i$
   occurs as the $r$'th element in $[d-1]'\sm g(A)$, so $i+1$ occurs
   as the $r$'th element in $\dl_i([d-1]'\sm g(A))$.  Then
   \[ (\dl_i)_*z''(\al\dl_i)=(-1)^{r-1}(\dl_i)_*(w_i)\wedge w''
        = (-1)^{r-1}(w_i+w_{i+1})\wedge w''.
   \]
   We next need to understand $\ep(\al\dl_i)$.  By definition we have 
   \[ \ep(\al\dl_i) u(\al\dl_i) \wedge z''(\al\dl_i) = \tht_{[d-1]}.
   \]
   As $\dl_if=g$ and
   $(\dl_i)_*z''(\al\dl_i)=(-1)^{r-1}(w_i+w_{i+1})\wedge w''$ we have
   $u(\al\dl_i)=u(\al)$ and 
   \[ (-1)^{r-1}\ep(\al\dl_i) u(\al) \wedge
        (w_i+w_{i+1})\wedge w'' = 
         \tht_{[d]\sm\{i\}}.
   \]
   We then multiply both sides on the left by $w_i$ to get 
   \[ (-1)^{m+r-1}\ep(\al\dl_i) u(\al)
        \wedge w_i\wedge w_{i+1}\wedge w'' = 
         (-1)^{i-1}\tht_{[d]}.
   \]
   After comparing this with the definition of $\ep(\al)$, we see that
   $\ep(\al\dl_i)=(-1)^{m+r+i}\ep(\al)$.  This gives 
   \begin{align*}
    (-1)^i(1\ot(\dl_i)_*)z(\al\dl_i)
     &= (-1)^i(-1)^{m+r+i}\ep(\al)z'(\al)\ot(-1)^{r-1}
       (w_i+w_{i+1})\wedge w'' \\
     &= (-1)^{m+1}\ep(\al)z'(\al)(dt_i\im z''(\al)) 
      = (-1)^{m+1}(1\ot\tau_i)(z(\al))
   \end{align*}
   as required.
 \end{itemize}
\end{proof}

\begin{corollary}
 The maps $\phi\:\tN_*(S^A\Smash X_+)\to\tH_*(S^A)\ot\Phi_*(X)$ and
 $\phi^\#\:U_*(A,X)\to\Phi_*(X)$ are chain maps.
\end{corollary}
\begin{proof}
 Lemma~\ref{lem-delta-z} is the universal example.  In more detail, we
 first note that $z(\al)$ involves only the exterior generators
 $dt_i$ so $(1\ot\dl')(z(\al))=0$ and
 \begin{align*} 
  (1\ot\dl)(z(\al)) &= (1\ot\dl'')(z(\al)) \\
   &= -\sum_j i_{[d]\sm\{j\}}(1\ot\tau_j)(z(\al)) \\
   &= (-1)^m \sum_j(-1)^j(1\ot(\dl_j)_*)(z(\al\dl_j)).
 \end{align*}
 Next, we will also write $\dl$ for the standard differential on
 $\tH_*(S^A)\ot\Phi_*(X)$, which is
 $\dl(a\ot b)=(-1)^{|a|}a\ot\dl(b)$.  This gives 
 \[ \dl(z(\al)) = (-1)^m(1\ot\dl)(z(\al)) = 
     \sum_j(-1)^j(1\ot(\dl_j)_*)(z(\al\dl_j)).
 \]

 Now consider an element $x\in X_d$, giving a map $\hx\:\Dl_d\to X$
 and thus a map $\hx_*\:\Phi_{[d],*}=\Phi_*(\Dl_d)\to\Phi_*(X)$.  If
 we apply the map $1\ot\hx_*$ to the above equation and use the
 naturality of $\dl$, the left hand side becomes
 $\dl(\phi(x,\al))$.   The right hand side becomes 
 $\sum_j(-1)^j(1\ot\widehat{d_jx}_*)(z(\al\dl_j))$, which is
 $\phi(d(x,\al))$.  This shows that $\phi$ is a chain map, and it
 follows adjointly that the same is true for $\phi^\#$.
\end{proof}

\begin{definition}\label{defn-nu}
 We define $\nu\:U_*(A,X)\ot U_*(B,Y)\to U_*(A\amalg B,X\tm Y)$ by
 applying the functor $\Hom(\tH_*(S^{A\amalg B}),-)$ to the composite
 \[ \tN_*(S^A\Smash X_+)\ot \tN_*(S^B\Smash Y_+) 
     \xra{\mu}
    \tN_*(S^A\Smash X_+\Smash S^B\Smash Y_+) 
     \xra{(1\Smash\tau\Smash 1)_*}
    \tN_*(S^{A\amalg B}\Smash (X\tm Y)_+)
 \]
 and using the isomorphism
 $\tH_*(S^A)\ot\tH_*(S^B)\to\tH_*(S^{A\amalg B})$.
\end{definition}

\begin{lemma}\label{lem-z-gamma}
 Suppose we have a shuffle $(\zt,\xi)\:[d+e]\to[d]\tm[e]$ and maps
 $\al\:[d]\to\Map(A,[1])$ and $\bt\:[e]\to\Map(B,[1])$ (with $|A|=m$
 and $|B|=n$).  Define $\gm\:[d+e]\to\Map(A\amalg B,[1])$ by
 $\gm_a(k)=\al_a(\zt(k))$ for $a\in A$ and $\gm_b(k)=\bt_b(\xi(k))$
 for $b\in B$.  Let $\lm$ denote the map 
 \[ \tH_*(S^A)\ot\Tht_{[d],*}\ot\tH_*(S^B)\ot\Tht_{[e],*} 
     \xra{1\ot\tau\ot 1}
    \tH_*(S^A)\ot\tH_*(S^B)\ot\Tht_{[d],*}\ot\Tht_{[e],*}
     \xra{\mu\ot\mu_{\zt\xi}}
    \tH_*(S^{A\amalg B}) \ot \Tht_{[d+e],*}.
 \]
 Then $z(\gm)=\sgn(\zt,\xi)\lm(z(\al)\ot z(\bt))$.
\end{lemma}
\begin{proof}
 If any $\al_a$ or $\bt_b$ fails to be surjective then so does the
 corresponding map $\gm_a$ or $\gm_b$, so both sides of the claimed
 equality are zero.  We ignore this case from now on.

 Put $f(a)=\al_a^\dag(1)$ and $g(b)=\bt_b^\dag(1)$ and
 $h(c)=\gm_c^\dag(1)$.  As $(\zt,\xi)$ is a shuffle we know that the
 maps $[d]'\xra{\zt^\dag}[d+e]'\xla{\xi^\dag}[e]'$ give a coproduct
 decomposition, and from the definitions we have $h(a)=\zt^\dag(f(a))$
 and $h(b)=\xi^\dag(g(b))$.  It follows that $h$ is injective iff both
 $f$ and $g$ are injective, and we may assume that this is the case as
 otherwise both sides of the claimed equality are zero.

 From our description of $h$, we have 
 \[ [d+e]'\sm h(A\amalg B) = 
     \zt^\dag([d]'\sm f(A)) \amalg \xi^\dag([e]'\sm g(B)),
 \]
 so $z''(\gm)=\pm \zt^\bullet(z''(\al))\wedge\xi^\bullet(z''(\bt))$.
 By a similar argument we have $z'(\gm)=\pm\mu(z'(\al)\ot z'(\bt))$
 and so $z(\gm)=\pm \lm(z(\al)\ot z(\bt))$.  The real issue is just to
 control the signs more precisely.  For this we note that 
 \[ \tht_{[d]} = \ep(\al) u(\al)\wedge z''(\al) 
     \hspace{5em}
    \tht_{[e]} = \ep(\bt) u(\bt)\wedge z''(\bt),
 \]
 so
 \[ \zt^\bullet\tht_{[d]}\wedge \xi^\bullet\tht_{[e]} = 
     \ep(\al)\ep(\bt)
      \zt^\bullet u(\al)\wedge 
      \zt^\bullet z''(\al) \wedge 
      \xi^\bullet u(\bt) \wedge 
      \xi^\bullet z''(\bt).
 \]
 After using Definition~\ref{defn-shuffle-sign} and reordering the
 factors, this gives 
 \[ \tht_{[d+e]} =
     (-1)^{n(d-m)} \sgn(\zt,\xi) \ep(\al)\ep(\bt) 
     \zt^\bullet u(\al)\wedge \xi^\bullet u(\bt) \wedge 
     \zt^\bullet z''(\al)\wedge \xi^\bullet z''(\bt).
 \]
 Now put 
 \begin{align*}
  U &= \K\{w_{h(c)}\st c\in A\amalg B\} \\
  V &= \K\{w_i\st\gm(i)=\gm(i-1)\} 
     = \K\{w_i\st i\not\in h(A\amalg B)\}
 \end{align*}
 as in Definition~\ref{defn-z-alpha}.  We find that 
 \begin{align*}
  \zt^\bullet u(\al)\wedge \xi^\bullet u(\bt) & \in U \\
  \zt^\bullet z''(\al)\wedge \xi^\bullet z''(\bt) & \in V, 
 \end{align*}
 so the above expression for $\tht_{[d+e]}$ can be used (together with
 the isomorphism $\tH_{m+n}(S^{A\amalg B})\simeq\Lm^{m+n}(U)$ induced
 by $h$) to calculate $z(\gm)$.  The result is 
 \[ z(\gm) = 
     (-1)^{n(d-m)} \sgn(\zt,\xi) \ep(\al)\ep(\bt) 
     \mu(z'(\al)\ot z'(\bt)) \ot 
     \zt^\bullet z''(\al)\wedge \xi^\bullet z''(\bt),
 \]
 and this is the same as $\lm(z(\al)\ot z(\bt))$.
\end{proof}

\begin{proposition}\label{prop-nu-phi}
 The following diagram commutes:
 \[ \xymatrix{
     U_*(A,X)\ot U_*(B,Y)
      \rto^{\nu}
      \dto_{\phi^\#\ot\phi^\#} &
     U_*(A\amalg B,X\tm Y) \dto^\phi \\
     \Phi_*(X)\ot\Phi_*(Y) \rto_{\mu} &
     \Phi_*(X\tm Y).
    }
 \]
\end{proposition}
\begin{proof}
 It will be enough to check commutativity of the adjoint diagram
 \[ \xymatrix{
     \tN_*(S^A\Smash X_+)\ot\tN_*(S^B\Smash Y_+) 
      \rto^\mu \dto_{\phi\ot\phi} &
     \tN_*(S^A\Smash X_+\Smash S^B\Smash Y_+) 
      \rto^{(1\Smash\tau\Smash 1)_*} &
     \tN_*(S^{A\amalg B}\Smash (X\tm Y)_+) \dto^\phi \\
     \tH_*(S^A)\ot\Phi_*(X)\ot\tH_*(S^B)\ot\Phi_*(Y) 
      \rto_{1\ot\tau\ot 1} &
     \tH_*(S^A)\ot\tH_*(S^B)\ot\Phi_*(X)\ot\Phi_*(Y) 
      \rto_{\mu\ot\mu} &
     \tH_*(S^{A\amalg B})\ot \Phi_*(X\tm Y).
    }
 \]
 Consider elements $\al\in(BA)_d$ and $x\in X_d$ and $\bt\in(BB)_e$
 and $y\in Y_e$.  The generator $(\al,x)\ot(\bt,y)$ maps to
 \[ \sum_{\zt,\xi}\sgn(\zt,\xi)(\al\zt,\bt\xi,\zt^*(x),\xi^*(y)) 
     \in \tN_{d+e}(S^{A\amalg B}\Smash(X\tm Y)_+).
 \]
 The term indexed by the shuffle $(\zt,\xi)$ then maps to
 $\sgn(\zt,\xi)(\zt^*(x),\xi^*(y))\ot z(\al\zt,\bt\xi)$ in
 $\tH_*(S^{A\amalg B})\ot\Phi_*(X\tm Y)$.  It follows from
 Lemma~\ref{lem-z-gamma} that the other route around the diagram
 yields the same result.
\end{proof}

\begin{definition}\label{defn-eta}
 Suppose we have a set $A$ with $|A|=m$.  We note that when $k>m$ we
 have $\ND_k(S^A)=\emptyset$ and so $\tN_k(S^A)=0$; this means that 
 \[ \tH_m(S^A)=\ker(d \: \tN_m(S^A) \to \tN_{m-1}(S^A))
     \leq \tN_m(S^A).
 \]
 The inclusion $\tH_m(S^A)\to\tN_m(S^A)$ gives a cycle in
 $U_0(A,1)=\Hom(\tH_m(S^A),\tN_m(S^A))$, which we denote by $\eta_A$.
\end{definition}

\begin{definition}
 Given an injective map $\lm\:A\to B$, we define 
 \[ \lm_*\:\Hom(\tH_*(S^A),\tN_*(S^A\Smash X_+)) \to 
     \Hom(\tH_*(S^B),\tN_*(S^B\Smash X_+))
 \]
 as follows.  Firstly, if $\lm$ is a bijection then we just transport
 the structure in the obvious way.  Next, suppose that $\lm$ is just
 the inclusion of a subset, so $B=A\amalg Z$ for some $Z$.  We then
 have a map
 \[ \nu \: U_*(A,X)\ot U_*(Z,1) \to U_*(A\amalg Z,X\tm 1) = U_*(B,X) \]
 and we put $\lm_*(u)=\nu(u\ot\eta_Z)$.  Finally, an arbitrary
 monomorphism can be written uniquely as $\lm=\lm_1\lm_0$, where
 $\lm_1$ is a subset inclusion and $\lm_0$ is a bijection.  We then
 put $\lm_*=(\lm_1)_*(\lm_0)_*$. 
\end{definition}

\begin{lemma}
 $\lm_*$ is a chain map and is functorial.
\end{lemma}
\begin{proof}
 Left to the reader.
\end{proof}

\begin{lemma}\label{lem-lm-phi}
 For any monomorphism $\lm\:A\to B$, the diagram
 \[ \xymatrix{
      U_*(A,X)
       \rrto^{\lm_*} \drto_{\phi^\#} & & 
      U_*(B,X)
       \dlto^{\phi^\#} \\
      & \Phi_*(X) 
    }
 \]
 commutes.
\end{lemma}
\begin{proof}
 This is clear if $\lm$ is an isomorphism, and is a special case of
 Proposition~\ref{prop-nu-phi} if $\lm$ is a subset inclusion.  The
 general case follows from these special cases.
\end{proof}

\begin{definition}\label{defn-U}
 We write $U_*(X)$ for the colimit of the complexes $U_*(A,X)$ as $A$
 runs over the category of finite sets an injective maps.  We let
 $\psi\:U_*(X)\to\Phi_*(X)$ denote the map induced by the maps
 $\phi^\#\:U_*(A,X)\to\Phi_*(X)$ (which exists by
 Lemma~\ref{lem-lm-phi}).  
\end{definition}

\begin{theorem}\label{thm-U-Phi}
 The map $\psi\:U_*(X)\to\Phi_*(X)$ is an isomorphism.
\end{theorem}

The proof will be given in several stages.  Firstly, the construction
given below immediately implies that $\psi$ is surjective.

\begin{construction}\label{cons-zt}
 Suppose we have $x\in\ND_n(X)$ and $\nu\:[n]\to\N$ and $J\sse[n]'$, 
 say $J=\{j_1<\dotsb<j_r\}$.  Put
 $w_J=w_{j_1}\wedge\dotsb\wedge w_{j_r}\in\Tht_{[n],r}$.
 We will construct an element $\zt(x,\nu,J)\in U_*(X)$ with
 $\psi(\zt(x,\nu,J))=x\ot t^{[\nu]}w_J\in\Phi_*(X)$. 

 First put $d=-1+\sum_{i=0}^n(\nu_i+1)$, and let $\sg\:[d]\to[n]$ be
 the unique nondecreasing surjective map such that
 $|\sg^{-1}(i)|=\nu_i+1$ for all $i$.  Put $A=[d]'\sm\sg^\dag(J)$ and
 $m=|A|$ and let $f\:A\to [d]'$ be the inclusion.  Define
 $\al\:[d]\to\Map(A,[1])$ by 
 \[ \al_a(i) =
      \begin{cases}
        0 & \text{ if } i < f(a) \\
        1 & \text{ if } f(a)\leq i.
      \end{cases}
 \]
 We find that $z''(\al)=w_{\sg^\dag(J)}$ and so (using
 Definition~\ref{defn-int-fibre}) we have
 $\sg_*(z''(\al))=t^{[\nu]}w_J$.  Now let 
 \[ \zt_1(x,\nu,J) \: \tH_m(S^A)\to\tN_d(S^A\Smash X) \]
 be the map that sends the generator $\ep(\al)z'(\al)$ to $(\al,x)$.
 Then $\zt_1(x,\nu,J)\in U_*(A,X)$ and
 $\phi^\#\zt_1(x,\nu,J)=x\ot t^{[\nu]}w_J$.  We also write
 $\zt(x,\nu,J)$ for the image of $\zt_1(x,\nu,J)$ in $U_*(X)$, so that
 $\psi(\zt(x,\nu,J))$.
\end{construction}

We next need the counterpart in $U_*(X)$ of the relation
$\sum_it_i=1$.  
\begin{lemma}\label{lem-zt-factor}
 In the notation of Construction~\ref{cons-zt} we have 
 \[ \sum_{i=0}^n(\nu_i+1)\zt(x,\nu+\dl_i,J)=\zt(x,\nu,J). \]
\end{lemma}
\begin{proof}
 We will freely use the notation of the above construction. 

 Put $A_+=A\amalg\{\infty\}$ so we have a class
 $\xi=\mu(\zt_1(x,\nu,J)\ot\eta_\{\infty\})\in U_*(A_+,X)$ which
 represents $\zt(x,\nu,J)$.  Now $\xi$ can be written as a sum of
 terms, one for each shuffle $(\lm,\rho)\:[d+1]\to[d]\tm[1]$.  Such a
 shuffle is determined by the number $k=\rho^\dag(1)\in[d+1]'$;
 indeed, $\lm$ is forced to be the unique map in $\EE([d+1],[d])$ that
 takes the value $k-1$ twice.  Define $\nu(k)\:[n]\to\N$ by
 $\nu(k)_i=|(\sg\lm)^{-1}\{i\}|-1$.  We find that the $k$'th term in
 the product $\mu(\zt_1(x,\nu,J)\ot\eta_\{\infty\})$ represents
 $\zt(x,\nu(k),J)$, and that there are $\nu_i+1$ different values of
 $k$ for which $\nu(k)=\nu+\dl_i$.  The claim follows.
\end{proof}

\begin{corollary}\label{cor-psi-split}
 There is a well-defined map 
 \[ \zt' \: \Phi_*(X) =
      \bigoplus_{x\in\ND_n(X)}\Tht_{[n],*} \to U_*(X)
 \]
 given by $\zt'(x\ot t^{[\nu]}w_J)=\zt(x,\nu,J)$.  Moreover, we have
 $\psi\zt'=1\:\Phi_*(X)\to\Phi_*(X)$. 
\end{corollary}
\begin{proof}
 The stated formula certainly defines a map
 \[ \bigoplus_{x\in\ND_n(X)}\tP_{[n]}\ot\Lm^*(W^\vee_{[n]})
      \to U_*(X).
 \]
 We simply need to pass from $\tP_{[n]}$ to
 $P_{[n]}=\tP_{[n]}/(1-\sum_it_i)$, and this is precisely what we get
 from Lemma~\ref{lem-zt-factor}.
\end{proof}

\begin{lemma}\label{lem-zt-psi}
 The composite $U_*(A,X)\xra{\phi^\#}\Phi_*(X)\xra{\zt'}U_*(X)$ is
 just the colimit inclusion map.
\end{lemma}
\begin{proof}
 Put $m=|A|$ and fix a generator $u\in\tH_m(S^A)$.  Given
 $v\in\tN_d(S^A\Smash X_+)$ we write $u^{-1}v$ for the element of
 $U_*(A,X)$ given by $u\mapsto v$.  The group 
 $U_d(A,X)$ is generated by elements $u^{-1}(\al,x)$ where
 $\al\:[d]\to\Map(A,[1])$ and $x\in X_d$ and the pair $(\al,x)$ is
 nondegenerate.  To avoid trivial cases, we may assume that each
 $\al_a\:[d]\to[1]$ is surjective, so we can define
 $f(a)=\al_a^\dag(1)$ as usual.  

 If $f$ is not injective, it is built into the definitions that  
 $z(\al)=0$ and so $\phi^\#(\al,x)=0$, so we must show that
 $u^{-1}(\al,x)$ also maps to zero in the colimit.  We can choose
 $a\neq a'$ with $f(a)=f(a')$, and let $\tau$ denote the transposition
 that exchanges $a$ and $a'$.  We find that $\tau_*(u)=-u$ but
 $\tau_*(\al,x)=(\al,x)$, so $\tau_*(u^{-1}(\al,x))=-u^{-1}(\al,x)$,
 which gives the required vanishing.

 From now on we assume that $f$ is injective.  As in
 Lemma~\ref{lem-degen-split} we can write $x=\sg^*(y)$ for some
 nondegenerate simplex $y\in X_n$ and some surjective map
 $\sg\:[d]\to[n]$.  To avoid further trivial cases, we may assume that
 the pair $(\al,x)$ is nondegenerate, which is equivalent to the
 condition $[d]'=f(A)\cup\sg^\dag([n]')$.  Define $\nu\:[n]\to\N$ by
 $\nu_i=|\sg^{-1}\{i\}|-1$, so that $\sg_*(1)=t^{[\nu]}$.  Put
 $J'=[d]'\sm f(A)$, so that $z''(\al)=w_{J'}$.  As
 $[d]'=f(A)\cup\sg^\dag([n]')$ we must have $J'=\sg^\dag(J)$ for some
 $J\sse [n]'$, and this implies that $J=\sg(J')$ and so
 $\sg_*(z''(\al))=t^{[\nu]}w_J$.  It follows that
 $\phi^\#(u^{-1}(\al,x))=\ep'x\ot t^{[\nu]}w_J$, where the sign
 $\ep'\in\{\pm 1\}$ is determined by the relation
 $z(\al)=\ep'u\ot z''(\al)$.  Now put $A'=f(A)$, so $f$ gives a
 bijection $A\to A'$ and thus an isomorphism $U_*(A,X)\to U_*(A',X)$.
 From the definitions we see that $\zt'\phi^\#(u^{-1}(\al,x))$ is
 represented by $\ep'\zt_1(x,\nu,J)\in U_*(A',X)$, which is just the
 image of $u^{-1}(x,\al)$ under this isomorphism.  The claim follows. 
\end{proof}

\begin{proof}[Proof of Theorem~\ref{thm-U-Phi}]
 Corollary~\ref{cor-psi-split} tells us that $\psi\zt'=1$, and
 Lemma~\ref{lem-zt-psi} implies that $\zt'\psi=1$.
\end{proof}

\appendix

\section{Recollections on the simplicial category}
\label{apx-simplicial}

In this section we recall some facts about the simplicial category.
Most of them are fairly standard but we will need to use the details
so it is convenient to give a self-contained account here.  Many of
these facts were first proved in~\cite{eizi:ssc} or~\cite{gazi:cfh};
the more recent book~\cite{frpi:cst} is also a useful reference.

\begin{definition}\label{defn-Delta}
 As usual, we let $\DDl$ denote the category whose objects are the
 finite ordered sets $[n]=\{0,\dotsc,n\}$, and whose morphisms are the
 nondecreasing maps.  All maps mentioned in this section are
 implicitly assumed to be nondecreasing.  We also write $\EE([n],[m])$
 for the subset of $\DDl([n],[m])$ consisting of surjective maps.
\end{definition}

\begin{definition}\label{defn-dagger}
 Given a surjective map $\al\:[n]\to[m]$, we define
 $\al^\dag\:[m]\to[n]$ by $\al^\dag(j)=\min\{i\st\al(i)=j\}$.  We also
 write $[n]'=[n]\sm\{0\}=\{1,\dotsc,n\}$ and note that $\al^\dag(0)=0$
 and $\al^\dag([m]')\sse[n]'$.
\end{definition}

\begin{lemma}\label{lem-dag}
 The map $\al^\dag$ is injective, and $\al\al^\dag=1$.  Moreover, if
 $\bt\:[n]\to[p]$ is another surjection then
 $(\bt\al)^\dag=\al^\dag\bt^\dag$.
\end{lemma}
\begin{proof}
 Left to the reader.
\end{proof}

\begin{definition}\label{defn-sg-A}
 We say that a subset $A\sse [n]$ is \emph{pointed} if $0\in A$. 
 Given a pointed subset $A\sse[n]$ with $|A|=m+1$, we let
 $\sg_A\:[m]\to[n]$ be the unique injection with
 $\sg_A([m])=A$, and we define $\pi_A\:[n]\to[m]$ by
 $\pi_A(i)=\max\{j\st\sg_A(j)\leq i\}$.  We also define
 $\ep_A=\sg_A\pi_A\:[n]\to[n]$, so
 $\ep_A(i)=\max\{j\in A\st j\leq i\}$ and $\ep_A^2=\ep_A$.
\end{definition}
\begin{lemma}\label{lem-sg-A}
 \begin{itemize}
  \item[(a)] Any injective map $\bt\:[m]\to[n]$ with
   $\bt(0)=0$ has the form $\bt=\sg_A$ for some (unique) pointed set $A$,
   namely $A=\bt([m])$.
  \item[(b)] Any surjective map $\al\:[n]\to[m]$ has the
   form $\al=\pi_A$ for some (unique) pointed set $A$, namely
   $A=\{0\}\cup\{i>0\st\al(i)>\al(i-1)\}$.
  \item[(c)] Let $\gm\:[n]\to[n]$ be a map with
   $i\geq\gm(i)=\gm^2(i)$ for all $i$.  Then $\gm=\ep_A$ for some
   (unique) pointed set $A$, namely $A=\{i\st\gm(i)=i\}$.
 \end{itemize}
\end{lemma}
\begin{proof}
 Left to the reader.
\end{proof}

\begin{lemma}\label{lem-A-B}
 Suppose we have pointed sets $A\sse B\sse [n]$ with $|A|=m+1$ and
 $|B|=p+1$.  Put $\al=\pi_A\sg_B\:[p]\to[m]$.  Then $\al$ is
 surjective and fits into a commutative diagram as follows.
 \[ \xymatrix{
      [m]
       \ar@{->>}[rr]^\al
       \ar@{ >->}[dr]^{\sg_B}
       \ddto_1 & &
      [p] \\
      & [n]
       \ar@{->>}[ur]^{\pi_A}
       \ar@{->>}[dl]^{\pi_B} \\
      [m] & &
      [p]
       \uuto_1 
       \ar@{ >->}[ul]^{\sg_A}
       \ar@{ >->}[ll]^{\al^\dag}
    }
 \]
 Moreover, $\al$ is bijective iff $\al=1$ iff $A=B$.
\end{lemma}
\begin{proof}
 Left to the reader.
\end{proof}

\begin{lemma}\label{lem-ep-ep}
 If $A$ and $B$ are pointed subsets of $[n]$ then
 $(\ep_A\ep_B)^N=\ep_{A\cap B}$ for $N\gg 0$.
\end{lemma}
\begin{proof}
 For any $i\in [n]$ we have a decreasing sequence
 \[ i \geq\ep_B(i)\geq\ep_A\ep_B(i) \geq
      \ep_B\ep_A\ep_B(i) \geq \dotsb \geq 0.
 \]
 Let $\gm(i)$ denote the eventual value of this sequence.  We find
 that for $N\gg 0$ we have $\gm=(\ep_A\ep_B)^N=\ep_B(\ep_A\ep_B)^N$,
 from which it follows that $\gm=\ep_A\gm=\ep_B\gm=\gm^2$ and
 $i\geq\gm(i)$.  It follows that $\gm=\ep_C$, where
 $C=\img(\gm)=\{i\st\gm(i)=i\}$.  As $\gm=\ep_A\gm=\ep_B\gm$ we see
 that $C=\img(\gm)\sse\img(\ep_A)\cap\img(\ep_B)=A\cap B$, but it is
 clear that $\gm$ is the identity on $A\cap B$ so $C=A\cap B$.
\end{proof}

\subsection{Degeneracy}

\begin{lemma}\label{lem-degenerate}
 Let $X$ be a simplicial set, and let $x$ be an $n$-simplex of $X$.
 Then the following conditions are equivalent.
 \begin{itemize}
  \item[(1)] $x=\al^*(y)$ for some non-injective map $\al\:[n]\to[m]$
   and some $y\in X_m$.
  \item[(2)] $x=\bt^*(z)$ for some surjective map $\bt\:[n]\to[p]$
   (with $p<n$) and some $z\in X_p$.
  \item[(3)] $x=\pi_A^*(y)$ for some proper pointed subset
   $A\subset[n]$ and some $y\in X_{|A|-1}$. 
  \item[(4)] $x=\ep_A^*(x)$ for some proper pointed subset
   $A\subset[n]$.
 \end{itemize}
 We write $\ND_n(X)$ for the set of nondegenerate $n$-simplices.
\end{lemma}
\begin{proof}
 It is clear that~(2) implies~(1), and we can prove the converse by
 factoring $\al$ as a surjection followed by an injection.
 Lemma~\ref{lem-sg-A}(b) tells us that~(2) is equivalent to~(3).
 Using the facts that $\ep_A=\sg_A\pi_A$ and $\pi_A\sg_A=1$ we see
 that~(3) is equivalent to~(4).
\end{proof}
\begin{definition}
 We say that $x$ is \emph{degenerate} if the above conditions hold.
 We write $\ND_n(X)$ for the set of nondegenerate $n$-simplices.
\end{definition}

The next result is known as the Eilenberg-Zilber lemma.
\begin{lemma}\label{lem-degen-split}
 There is a canonical bijection
 $\psi\:\coprod_m\EE([n],[m])\tm\ND_m(X)\to X_n$ given by
 $\psi(\al,y)=\al^*(y)$.  
\end{lemma}
\begin{proof}
 Given $x\in X_n$, let $\CA$ denote the collection of pointed subsets
 $A\sse[n]$ such that $x=\ep_A^*(x)$.  Using the facts that
 $\ep_A=\sg_A\pi_A$ and $\pi_A\sg_A=1$ we see that
 $\CA=\{A\st x\in\img(\pi_A^*)\}$.  

 It is clear that $[n]\in\CA$, and Lemma~\ref{lem-ep-ep} implies that
 $\CA$ is closed under intersections, so $\CA$ has a smallest element,
 say $A$.  Put $m=|A|-1$ and $y=\sg_A^*(x)\in X_m$ and note that
 $x=\pi_A^*(y)$.  

 Suppose that $y=\bt^*(z)$ for some surjection $\bt\:[m]\to[p]$.  Then
 $\bt\pi_A=\pi_B$ for some $B\sse A$, but $x=\pi_B^*(Z)$ so $B\in\CA$
 so $A\sse B$.  It follows that $A=B$ and $p=m$ and $\bt=1$ so $y=z$.
 Using this we see that $y$ is nondegenerate.

 More generally, suppose we also have $x=\pi_B^*(z)$ for some $B$
 (\emph{a priori} unrelated to $A$) and $z\in X_p$ (\emph{a priori}
 unrelated to $y$).  Then again $B\in\CA$ so $A\sse B$ so we can apply
 Lemma~\ref{lem-A-B}: the map $\al=\pi_A\sg_B\:[p]\to[m]$ is
 surjective and satisfies $\al\pi_B=\pi_A$.  As $x=\pi_B^*(z)$ we have
 $z=\sg_B^*(x)=\sg_B^*\pi_A^*(y)=\al^*(y)$.  If $z$ is nondegenerate
 it follows that we must have $p=m$ and $\al$ must be the identity so
 $A=B$ and $y=z$.  Using this we see that $\psi$ is a bijection.
\end{proof}

\subsection{Shuffles}

We now recall some theory of shuffles.

\begin{definition}\label{defn-shuffle}
 Given a sequence $\un{n}=(n_1,\dotsc,n_r)\in\N^r$ with $\sum_in_i=n$,
 an \emph{$\un{n}$-shuffle} means a system of surjective maps
 $\zt_i\:[n]\to[n_i]$ such that the combined map
 $\zt\:[n]\to\prod_i[n_i]$ is injective.  We write $\Sg(\un{n})$ for
 the set of all $\un{n}$-shuffles.
\end{definition}
\begin{remark}
 We will most often need the case $r=2$.  An $(n,m)$-shuffle is then a
 pair of surjections $[n]\xla{\zt}[n+m]\xra{\xi}[m]$ such that the map
 $(\zt,\xi):[n+m]\to[n]\tm[m]$ is injective.
\end{remark}

\begin{lemma}\label{lem-shuffle-coprod}
 Let $\un{n}$ and $n$ be as above, and suppose we have sets
 $A_1,\dotsc,A_r\sse[n]'=\{1,\dotsc,n\}$ with $|A_i|=n_i$ and we put
 $\zt_i=\pi_{A_i\cup\{0\}}\:[n]\to[n_i]$.  Then the list $\un{\zt}$ is
 an $\un{n}$-shuffle iff $[n]'=\coprod_iA_i$.
\end{lemma}
\begin{proof}
 From the definition of $\pi_{A_i\cup\{0\}}$ we see that
 $A_i=\{s\in[n]'\st\zt_i(s)>\zt_i(s-1)\}$, so that
 $\bigcup_iA_i=\{s\in[n]'\st\zt(s)\neq\zt(s-1)\}$.  Thus $un{\zt}$ is
 a shuffle iff $\zt$ is injective iff $\bigcup_iA_i=[n]'$, and if this
 happens then the union is automatically disjoint by counting.
\end{proof}
\begin{corollary}
 We have $|\Sg(\un{n})|=n!/\prod_in_i!$.  In particular,
 $|\Sg(n,m)|=(n+m)!/n!m!$. \qed
\end{corollary}

\begin{lemma}\label{lem-operad}
 There are natural bijections
 \[ \Sg(m+n,p) \tm \Sg(m,n) \xra{L} \Sg(m,n,p) 
     \xla{R} \Sg(n,m+p) \tm \Sg(n,p)
 \]
 given by $L(\zt,\xi;\phi,\psi)=(\phi\zt,\psi\zt,\xi)$ and 
 $R(\zt,\xi;\phi,\psi)=(\zt,\xi\phi,\xi\psi)$.
\end{lemma}
\begin{proof}
 We will only discuss $L$; the case of $R$ is similar.

 Suppose that $(\zt,\xi)\in\Sg(m+n,p)$ and $(\phi,\psi)\in\Sg(m,n)$.
 Then $\zt$, $\xi$, $\phi$ and $\psi$ are all surjective, so the same
 is true of $\phi\zt$ and $\psi\zt$.  The maps
 $(\phi,\psi)\tm 1\:[m+n]\tm[p]\to[m]\tm[n]\tm[p]$ and
 $(\zt,\xi)\:[m+n+p]\to[m+n]\tm[p]$ are injective, so the same is true
 of their composite, so $L(\zt,\xi;\phi,\psi)\in\Sg(m,n,p)$.  Next,
 observe that to give a three-piece splitting
 $[m+n+p]'=A\amalg B\amalg C$ (with $|A|=m$ and $|B|=n$ and $|C|=p$)
 is the same as to give a splitting $[m+n+p]'=U\amalg C$ (with
 $|U|=m+n$ and $|C|=p$) together with a splitting $U=A\amalg B$ (with
 $|A|=m$ and $|B|=n$).  Using this together with the correspondence
 $T\leftrightarrow\pi_T$ we obtain a bijection
 $L'\:\Sg(m+n,p)\tm\Sg(m,n)\to\Sg(m,n,p)$.  We leave it to the reader
 to check that $L=L'$.
\end{proof}

\section{Integration over simplices}
\label{apx-int}

Recall that the map $\int_I\:\tP_I\to\K$ is defined by 
\[ \int_I t^\nu = 
     \left(\prod_i \nu_i!\right)/(n+\sum_i\nu_i)!,
\]
(where $n=|I|-1$) or equivalently $\int_I t^{[\nu]}=1/(n+|\nu|)!$.

\begin{lemma}\label{lem-int-well-defined}
 The map $\int_I\:\tP_I\to\K$ factors through $P_I$.
\end{lemma}
\begin{proof}
 We must show that $\int_I$ annihilates the ideal generated by
 $1-\sum_it_i$, or equivalently that
 $\int_It^{[\nu]}=\sum_i\int_It_it^{[\nu]}$.  We have
 $t_it^{[\nu]}=(1+\nu_i)t^{[\dl_i+\nu]}$, where $\dl_i\:I\to\N$ is the
 Kronecker delta, so $|\dl_i+\nu|=1+|\nu|$.  We thus have
 \begin{align*}
  \sum_i\int_It_it^{[\nu]}
   &= \sum_i (1+\nu_i) \int_I t^{[\nu+\dl_i]} 
    = \frac{1}{(n+1+|\nu|)!}\sum_i (1+\nu_i) \\
   &= \frac{n+1+|\nu|}{(n+1+|\nu|)!} = \frac{1}{(n+|\nu|)!} 
    = \int_I t^{[\nu]}
 \end{align*}
 as required.
\end{proof}

\begin{lemma}\label{lem-int-real}
 If $\K=\R$ then $\int_I f$ is just the integral of $f$ over the
 simplex $\Dl_I=\{x\:I\to\R_+\st \sum_ix_i=1\}$, with respect to the
 usual Lebesgue measure normalised so that $\mu(\Dl_I)=1/(|I|-1)!$.
\end{lemma}
\begin{proof}
 We may assume that $I=\{0,\dotsc,n\}$ and work by induction on $n$.
 We can identify $\Dl_n$ by projection with
 $\Dl'_I=\{x\in\R^n\st\sum_{i=1}^nx_i\leq 1\}$.  Define  
 \[ \int'_I f = 
   \int_{\Dl'_I} f(1-\sum_{i=1}^n x_i,x_1,x_2,\dotsc,x_n) 
     dx_1 \dotsb dx_n.
 \]
 We will show that $\int_It^{[\nu]}=\int'_It^{[\nu]}$ for any
 multiindex $\nu$ with $\nu_0=0$.  This will suffice because
 $P_I=\R[t_1,\dotsc,t_n]$.  When $n=0$ the claim is trivial.
 When $n=1$, the claim says that $\int_{t=0}^1t^{[n]}=1/(1+n)!$, which
 is also trivial.  This implies that $\int=\int'$ even on polynomials
 that are not in our preferred basis, which gives
 \[ \int_{t=0}^1(1-t)^{[i]}t^{[j]}= 1/{(1+i+j)!}; \]
 this will be useful later.

 For $n>0$ we define a map $\phi\:\Dl'_{n-1}\tm[0,1]\to\Dl'_n$ by
 $\phi(t,s)=(st,1-s)$.  This is bijective away from a set of measure
 zero, and  the Jacobian is $s^{n-1}$.  Given a multiindex
 $\nu=(0,\nu_1,\dotsc,\nu_n)$, write $\nu'$ for the truncated sequence
 $(0,\nu_1,\dotsc,\nu_{n-1})$.  We then have 
 \[ \phi(t,s)^{[\nu]} = (1-s)^{[\nu_n]} (ts)^{[\nu']} 
     = (1-s)^{[\nu_n]}s^{|\nu'|} t^{[\nu']}.
 \]
 We may assume inductively that 
 \[ \int'_{[n-1]}t^{[\nu']} = \frac{1}{(n-1+|\nu'|)!}, \]
 so 
 \begin{align*}
  \int'_{[n]} t^{[\nu]} 
   &= \int_{s=0}^1\int'_{[n-1]}\phi(x,s)^{[\nu]} s^{n-1} ds \\
   &= \int_{s=0}^1 (1-s)^{[\nu_n]}s^{n-1+|\nu'|} ds 
       \int'_{[n-1]} x^{[\nu']} \\
   &= \int_{s=0}^1 (1-s)^{[\nu_n]} s^{[n-1+|\nu'|]} ds 
    = \frac{1}{(1+\nu_n+n-1+|\nu'|)!} = \frac{1}{(n+|\nu|)!},
 \end{align*}
 as required.

 Now $\int'_If$ is certainly the integral of $f$ over $\Dl_I$ with
 respect to some normalisation of Lebesgue measure.  To determine the
 normalisation, note that $\int'_I1=\int_It^{[0]}=1/n!$ as required.
\end{proof}

\begin{lemma}\label{lem-int-s}
 Take $I=[n]=\{0,1,\dotsc,n\}$, use the parameters $s_k=\sum_{j<k}t_j$
 for $k=1,\dotsc,n$.  Consider a monomial
 $s^\nu=\prod_{k=1}^ns_k^{\nu_k}$.
 Put $\mu_k=\sum_{j\leq k}(\nu_j+1)$ and $\mu=\prod_i\mu_i$.  Then
 $\int_{[n]}s^\nu=1/\mu$.
\end{lemma}
\begin{proof}
 It will suffice to prove this when $\K=\R$, in which case we have
 $\int_Is^\nu=\int'_Is^\nu$.  By a straightforward change of variables
 this becomes
 \[ \int_Is^\nu =
     \int_{0\leq s_1\leq\dotsb\leq s_n\leq 1} s^\nu ds_1\dotsb ds_n.
 \]
 Suppose that the lemma holds for some $n$.  Using the change of
 variables $s_i\mapsto rs_i$ (which has Jacobian $r^n$) we see that
 \[ \int_{0\leq s_1\leq\dotsb\leq s_n\leq r} s^\nu ds_1\dotsb ds_n
     = r^{n+\sum_i\nu_i} 
       \int_{0\leq s_1\leq\dotsb\leq s_n\leq r} s^\nu ds_1\dotsb ds_n
     = r^{\mu_n}/\mu.
 \]
 Now multiply by $r^m$ and integrate from $r=0$ to $r=1$; the right
 hand side becomes $1/((m+1+\mu_n)\mu)$.  Now change notation,
 replacing $r$ by $s_{n+1}$ and $m$ by $\nu_{n+1}$; this gives the
 case $n+1$ of the lemma.
\end{proof}

\begin{lemma}\label{lem-int-prod}
 Suppose that $f\in P_{[n]}$ and $g\in P_{[m]}$.  Then 
 \[ \int_{[n]}f \cdot \int_{[m]} g = 
     \sum_{(\al,\bt\in\Sg(n,m))} \int_{[n+m]} \al^*(f)\bt^*(g).
 \]
 (Here $\Sg(n,m)$ is the set of $(n,m)$-shuffles, as in
 Definition~\ref{defn-shuffle}.) 
\end{lemma}
\begin{proof}
 Put $\Dl''_n=\{s\in\R^n\st 0\leq s_1\leq\dotsb\leq s_n\leq 1\}$.  As
 implicitly used in the proof of the previous lemma, there is a
 homeomorphism $\Dl_{[n]}\to\Dl''_n$ given by 
 \[ t \mapsto (t_0,t_0+t_1,\dotsc,\sum_{i<n}t_i). \]
 Now consider a shuffle $(\al,\bt)\in\Sg(n,m)$, and the corresponding
 maps 
 \[ \{1,\dotsc,n\} \xra{\phi} \{1,\dotsc,n+m\} 
     \xla{\psi} \{1,\dotsc,m\}
 \]
 given by $\phi(j)=\min\{i\st\al(i)=j\}$ and
 $\psi(k)=\min\{i\st\bt(i)=k\}$.  These give a map
 $(\al_*,\bt_*)\:\Dl''_{n+m}\to\Dl''_n\tm\Dl''_m$, with
 $\al_*(s)_i=s_{\phi(i)}$ and $\bt_*(s)_j=s_{\psi(j)}$.  Let
 $X_{\al\bt}$ be the image of this map.  It is standard 
 that these are the top-dimensional simplices in a triangulation of
 $\Dl''_n\tm\Dl''_m$, so 
 \[ \int_{[n]}f \cdot \int_{[m]} g = 
     \sum_{\al,\bt} \int_{X_{\al\bt}} f\ot g. 
 \]
 Moreover, from the form of the maps $\al_*$ and $\bt_*$ it is clear
 that the Jacobian of $(\al_*,\bt_*)\:\Dl''_{n+m}\to\Dl''_n\tm\Dl''_m$
 is one.  The lemma follows.
\end{proof}

\end{document}